\documentclass[11pt, twoside, leqno]{amsart}  
\usepackage{lipsum}
\usepackage{amsfonts}
\usepackage{graphicx}
\usepackage{epstopdf}
\usepackage{algorithmic}
\usepackage{calligra}
\usepackage[dvipsnames]{xcolor}
\usepackage{amsfonts,amsmath,amsthm,amssymb}
\usepackage{mathtools}
\usepackage{hyperref}
\usepackage[makeroom]{cancel}
\usepackage{autonum}
\usepackage{hhline}
\usepackage{array}
\usepackage{diagbox}
\usepackage{tcolorbox}
\usepackage{mdframed}
\usepackage{multicol}
\usepackage{graphicx}
\usepackage{subcaption}
\usepackage{moreverb}
\usepackage{bbm}
\usepackage[margin=1.38in]{geometry}
\usepackage{todonotes}
\usepackage{scalerel,amssymb}
\allowdisplaybreaks
\usepackage{mathrsfs}  
\usepackage{lineno}
\usepackage{todonotes}
\usepackage{tikz}
\usepackage{appendix}
\usepackage{enumitem}
\usepackage{pgfplots}
\usetikzlibrary{arrows.meta}
\usepackage{siunitx}
\usepackage[numbers,sort&compress]{natbib}
\definecolor{mygreen}{HTML}{43a047}
\usepackage{subcaption}
\usepackage{doi}

\newcommand{\Om}{\Omega}
\newcommand{\D}{\Delta}

\newcommand{\rhob}{\rho_{\textup{b}}}
\newcommand{\rhoa}{\rho_{\textup{a}}}
\newcommand{\Ca}{C_{\textup{a}}}
\newcommand{\Cb}{C_{\textup{b}}}
\newcommand{\kappaa}{\kappa_{\textup{a}}}
\newcommand{\Thetaa}{\Theta_{\textup{a}}}

\newcommand{\ddt}{\frac{\textup{d}}{\textup{d}t}}

\newcommand{\dt}{\, \textup{d} t}
\newcommand{\ds}{\, \textup{d} s }

\newcommand{\R}{\mathbb{R}} 
 
\newcommand{\Honetwo}{{H_\diamondsuit^2(\Omega)}}
\newcommand{\Honethree}{{H_\diamondsuit^3(\Omega)}}

\newtheorem{theorem}{Theorem}
\newtheorem{lemma}{Lemma}
\newtheorem{proposition}{Proposition}
\newtheorem{assumption}{Assumption}

\numberwithin{lemma}{section}
\numberwithin{proposition}{section}
\numberwithin{theorem}{section}
\numberwithin{equation}{section}
\makeatletter
\newcommand{\leqnomode}{\tagsleft@true}
\newcommand{\reqnomode}{\tagsleft@false}
\makeatother

\definecolor{grey}{rgb}{0.5,0.5,0.5}

\title[Local well-posedness of a coupled Westervelt--Pennes model]{Local well-posedness of a coupled \\[1mm] Westervelt--Pennes model of nonlinear ultrasonic heating}    
\subjclass[2010]{35L70, 35K05}      
   
\keywords{ultrasonic heating, Westervelt's equation,   nonlinear acoustics, Pennes bioheat equation, HIFU}  
    
\author{Vanja Nikoli\'c$^\dagger$}
\thanks{$^\dagger$Department of Mathematics,
	Radboud University, 
	Heyendaalseweg 135,
	6525 AJ Nijmegen, The Netherlands (\href{vanja.nikolic@ru.nl}{vanja.nikolic@ru.nl})}
\author{Belkacem Said-Houari$^\ddag$}
\thanks{$^\ddag$Department of Mathematics, College of Sciences, University of
	Sharjah, P. O. Box: 27272, Sharjah, United Arab Emirates (\href{bhouari@sharjah.ac.ae}{bhouari@sharjah.ac.ae})}
\begin{document}
	\vspace*{8mm}
	\begin{abstract}
		High-Intensity Focused Ultrasound (HIFU) waves are known to induce localized heat to a targeted area during medical treatments. In turn, the rise in temperature influences their speed of propagation. This coupling affects the position of the focal region as well as the achieved pressure and temperature values. In this work, we investigate a mathematical model of nonlinear ultrasonic heating based on the Westervelt wave equation coupled to the  Pennes bioheat equation that captures this so-called thermal lensing effect. We prove that this quasi-linear model is well-posed locally in time and does not degenerate under a smallness assumption on the pressure data.
					\end{abstract}   
	\vspace*{-7mm}  
	\maketitle             

\section{Introduction}
High-Intensity Focused Ultrasound (HIFU) is an innovative medical tool that relies on focused sound waves to induce localized heating to the targeted tissue~\cite{ter2016hifu}. Due to its non-invasive nature and relatively brief treatment time, it has excellent potential to be used in the therapy of various benign and malignant  tumors; see, e.g.,~\cite{maloney2015emerging, wu2003randomised, li2010noninvasive, hsiao2016clinical, hahn2018high}. The ability to accurately determine the properties of the pressure and temperature field in the focal region is crucial in these procedures and motivates the research into the validity of the corresponding mathematical models.\\
\indent It is well-known that the heating of tissue influences the speed of propagation of sound waves and, in turn, the position of the focal region; this effect is commonly referred to as thermal lensing~\cite{connor2002bio, hallaj1999fdtd, hallaj2001simulations}. In this work, we analyze a mathematical model of nonlinear ultrasonic heating that captures this effect. More precisely, we study a coupled problem consisting of the Westervelt wave equation of nonlinear acoustics~\cite{westervelt1963parametric}:
	\begin{equation}\label{p_Eq}
		p_{tt}-c^2(\Theta)\Delta p - b \Delta p_t = k(\Theta) \left(p^2\right)_{tt}
	\end{equation}
	and the Pennes bioheat equation~\cite{pennes1948analysis}:
		\begin{equation} \label{Heat_Eq}
	\rhoa \Ca\Theta_t -\kappaa\Delta \Theta+ \rhob \Cb W(\Theta-\Thetaa) = \mathcal{Q}(p_t),
	\end{equation}
	where $p$ is the acoustic pressure, $\Theta$ the temperature, and $\mathcal{Q}(p_t)$ is the absorbed acoustic energy. We refer to Section~\ref{Sec:ProblemSetting} below for further details on the modeling and the involved material parameters.\\
	\indent To the best of our knowledge, this is the first work dealing with a rigorous mathematical  analysis of a coupled Westervelt--Pennes model. Westervelt's equation is a quasilinear strongly damped (for $b>0$) wave equation that has been extensively studied by now in various settings with constant material parameters; see, e.g.,~\cite{kaltenbacher2009global, kaltenbacher2011well, meyer2011optimal, kaltenbacher2019well, kaltenbacher2020parabolic} and the references given therein, where results concerning  local well-posedness, global well-posedness, and asymptotic behavior of the solution have been established. The results on the well-posedness of the  Westervelt equation with an additional strong nonlinear damping and with $L^\infty(\Omega)$ varying coefficients have been obtained in~\cite{brunnhuber2014relaxation, nikolic2015local}. We mention that this wave equation  can also be rigorously recovered in the limit of a third-order nonlinear acoustic equation for vanishing thermal relaxation time; see the analysis in~\cite{bongarti2021vanishing, kaltenbacher2019jordan}. \\ 
	\indent A distinguishing feature of the present quasilinear thermo-acoustic problem is the dependence of propagation speed on the temperature, which we assume to be polynomial (in accordance with the real-world setting) and non-degenerate. Having temperature-dependent medium parameters presents a challenge due to the higher regularity the pressure should possess to tackle the nonlinearities. To resolve this issue, our theoretical approach relies on higher-order energy analysis of a suitable linearization combined with a fixed-point argument, under the assumption of smooth and small (with respect to pressure) data. Although the heat equation \eqref{Heat_Eq} has regularizing properties,  it does not seem feasible to transfer these to the pressure equation \eqref{p_Eq} and make use of the damping property of heat conduction, as in the classical thermo-elastic systems; see, e.g.,~\cite{racke2000evolution, lasiecka2017global, lasiecka2019long} and the references given therein. This issue arises due to the very weak nonlinear coupling in the present model. \\
	 \indent  A critical step in the analysis is handling the higher-order time-derivative of the pressure in the nonlinear term; that is, $k(\Theta) \left(p^2\right)_{tt}$. Due to the temperature-dependent coefficients, we have to rely on higher-order energies compared to the analysis of Westervelt equation in homogeneous media in~\cite{kaltenbacher2009global} and assume
	 \[
	 (p, p_t)_{\vert t=0} =(p_0, p_1)\in H^3(\Omega) \times H^2(\Omega).
	 \] 
	 Another well-known difficulty in the analysis of the Westervelt equation is related to the possible degeneracy for large values of $p$ since the factor in front of $p_{tt}$ in \eqref{p_Eq} can be written as $1-2k(\Theta)p$. This invokes the condition $1-2k(\Theta)p>0$ almost everywhere, which in turn requires $ \|p\|_{L^\infty}$ to remain small enough in time. The issue is commonly resolved by using a Sobolev embedding under the assumption of small pressure data, e.g., $H^2(\Omega)\hookrightarrow L^\infty(\Omega)$, as in~\cite{kaltenbacher2009global}. \\
	\indent We organize the rest of our exposition as follows. We provide more detailed insight into mathematical bio-acoustic modeling in Section~\ref{Sec:ProblemSetting}. Section~\ref{Sec:LinProblem} focuses on the energy analysis of a (partially) linearized uncoupled problem. In Section~\ref{Sec:NonLinProblem}, we present the study of the coupled nonlinear model by relying on the result from the previous section and Banach's fixed-point theorem. Our main well-posedness result is contained in Theorem~\ref{Thm:NlWellP}. We conclude the paper with a discussion and an outlook on future work.	
	\section{A model of ultrasonic heating based on the Westervelt equation} \label{Sec:ProblemSetting}
	Volume coupling of the acoustic pressure $p$ to the temperature field $\Theta$ is achieved via appropriate source terms and the use of temperature-dependent acoustic material parameters; see, e.g., \cite{cavicchi1984heat, norton2016westervelt, connor2002bio, hallaj1999fdtd, shevchenko2012multi}. In this work, we model the soft tissue as a thermoviscous fluid and consider the Westervelt equation in pressure form with the temperature-dependent speed of sound $c=c(\Theta)$:   
	\begin{equation} \label{Westervelt}
	\begin{aligned}
	p_{tt}-c^2(\Theta)\Delta p - b \Delta p_t = k(\Theta) \left(p^2\right)_{tt},
	\end{aligned}
	\end{equation}
	where the right-hand side coefficient is given by
	\begin{align} \label{k_q_relation}
	k(\Theta)=\frac{1}{\rho c^2(\Theta)}\beta_{\textup{acou}}.
	\end{align}
	Here, $\rho$ is the medium density and $\beta_{\textup{acou}}$ the acoustic coefficient of nonlinearity. The third term in \eqref{Westervelt} models thermoviscous losses. The parameter $b$ is known as the sound diffusivity~\cite{lighthill1956viscosity}. Assuming harmonic excitation with angular frequency $\omega$, it is connected to the absorption coefficient $\alpha$ via
	\begin{equation}  
	\begin{aligned}
	b=\frac{\alpha c_{\textup{a}}^3}{\omega^2},
	\end{aligned}    
	\end{equation}
where $c_{\textup{a}}$ is the ambient speed of sound (in the tissue); cf.~\cite{norton2016westervelt}. Equation \eqref{Westervelt} is obtained from the model considered in~\cite{connor2002bio} under the assumption of constant medium density and upon approximating $\dfrac{b}{c^2}p_{ttt}$ by $b \Delta p_t$. Note that if the attenuation obeys a frequency power law, equation \eqref{Westervelt} generalizes to involve a fractional damping term; see, e.g.,~\cite{norton2016westervelt}. This case is thus of interest for future analysis as well, but outside the scope of the current work. \\
\indent The temperature distribution in the tissue is modeled by the Pennes bioheat equation~\cite{pennes1948analysis} with a nonlinear source term:
	$$ 
	\rhoa \Ca\Theta_t -\kappaa\Delta \Theta+ \rhob \Cb W(\Theta-\Thetaa) = \mathcal{Q}(p_t).
	$$ 
	The function $\mathcal{Q} = \mathcal{Q}(p_t)$ represents the acoustic energy absorbed by the tissue at any given point. The term $\rhob \Cb W(\Theta-\Thetaa)$ models the removal of heat by blood circulation. Here, $\rhob$ and $\Cb$ are the density and specific heat capacity of blood,    respectively, and $W$ is the volumetric perfusion rate of the tissue measured in milliliters of blood per milliliter of tissue per second. The values of typical material properties in the human tissue can be found, for example, in~\cite[Table 3]{connor2002bio}. The coefficients $\rhoa$ and $\kappaa$ denote the ambient density and thermal conductivity (i.e., the tissue density and thermal conductivity). $\Ca$ is the ambient heat capacity and $\Thetaa$ is the ambient temperature. In the body, the latter is usually taken to be $\SI{37}{\celsius}$; see~\cite{connor2002bio}. \\
	\indent Altogether, we end up with the following coupled problem:
	\begin{subequations} \label{coupled_problem}
			\begin{equation} \label{coupled_problem_eq}
	\left\{ \begin{aligned}
	&p_{tt}-q(\Theta)\Delta p - b \Delta p_t = k(\Theta)\left(p^2\right)_{tt}, \qquad &&\text{in} \ \Omega \times (0,T), \\[1mm]
	& 	\rhoa \Ca\Theta_t -\kappaa\Delta \Theta+ \rhob \Cb W(\Theta-\Thetaa) = \mathcal{Q}(p_t), \qquad &&\text{in} \ \Omega \times (0,T),
	\end{aligned} \right.
	\end{equation}
where we have introduced the function \[q(\Theta)=c^2(\Theta).\]
For simplicity, we consider \eqref{coupled_problem_eq} together with homogeneous Dirichlet boundary conditions  
\begin{equation} \label{coupled_problem_BC}
p\vert_{\partial \Om}=0, \qquad \Theta\vert_{\partial \Om}=0,  
\end{equation}
and the initial data
\begin{equation} \label{coupled_problem_IC}
(p, p_t)\vert_{t=0}= (p_0, p_1), \qquad \Theta \vert_{t=0}=\Theta_0.
\end{equation}
\end{subequations}
The constant medium parameters appearing in \eqref{coupled_problem} are all assumed to be positive. The speed of sound $c=c(\Theta)$ typically exhibits polynomial dependence on the temperature. In water, for instance, it is taken to be
\[
\begin{aligned}
c(\Theta)=&\, \begin{multlined}[t]1402.39 + 5.0371\,\Theta - 5.8085 \times 10^{-2} \Theta^2 + 3.3420 \times 10^{-4} \Theta^3\\ - 1.4780 \times 10^{-6} \Theta^4
+ 3.1464 \times 10^{-9} \Theta^5; \end{multlined}
\end{aligned}
\] 
see~\cite[\S 2.2]{connor2002bio} and \cite{bilaniuk1993speed}. We thus make the following assumptions on the function $q$ in our analysis. Note that throughout the paper, we use $x \lesssim y$ to denote $x \leq C y$, where $C>0$ is a generic constant that may depend on $\Omega$, the final time $T$, and medium parameters.
\begin{assumption}\label{Assumption_1}
Let $q \in C^2(\R)$. We assume that there exists $q_0>0$, such that
	\[
	q(s) \geq q_0  \qquad \forall s \in \R.
	\]
Furthermore, there exist $\gamma_1\geq 0$ and $C_1>0$, such that
\begin{equation}
\begin{aligned}
|q''(s)|\leq \, C_1(1+|s|^{\gamma_1}) \qquad \forall s \in \R.
\end{aligned}
\end{equation}
\end{assumption}
\noindent By these assumptions and Taylor's formula, it further follows that 
\begin{equation}\label{q_prime_Assumption}
|q'(s)| \lesssim 1+|s|^{\gamma_1+1}.
\end{equation}
\noindent The function $k$ is assumed to be related to $q$ via \eqref{k_q_relation} throughout this work. Therefore, we have
\begin{equation}\label{Assumption_k_1}
\begin{aligned}
|k(\Theta)| \lesssim&\, \frac{1}{q_0}.
\end{aligned}
\end{equation}
Furthermore, since
\begin{equation}
\begin{aligned}
|k'(\Theta)| \lesssim&\, \frac{1}{q_0^2}|q'(\Theta)| \lesssim \frac{1}{q_0^2}(1+|\Theta|^{\gamma_1+1}), \\
|k''(\Theta)| \lesssim&\, \frac{1}{q_0^2}|q''(\Theta)|+\frac{1}{q_0^3}|q'(\Theta)|^2 
\lesssim\, \frac{1}{q_0^2}(1+|\Theta|^{\gamma_1})+\frac{1}{q_0^3}(1+|\Theta|^{\gamma_1+1})^2,
\end{aligned}
\end{equation}
we conclude that there exists $\gamma_2>0$, such that
\begin{equation}\label{Assumption_k_2}
\begin{aligned}
|k'(\Theta)| \lesssim\,1+|\Theta|^{\gamma_2+1},\qquad |k''(\Theta)| \lesssim\, 1+|\Theta|^{\gamma_2}.
\end{aligned}
\end{equation}
\subsection*{Modeling the absorbed acoustic energy} The acoustic energy absorbed by the tissue is represented by the source term $\mathcal{Q}=\mathcal{Q}(p_t)$ in the heat equation. We will make the following general assumptions concerning its properties in our analysis, which allow us to cover important particular cases from the literature.
\begin{assumption} \label{Q_assumption}
The mapping $Q$ is Lipschitz continuous on bounded subsets of the space $L^{\infty}(0,T;L^\infty(\Omega))$ with values in $L^2(0,T; L^2(\Omega))$, that is,
\begin{equation}\label{Lipschitz_Assumption}
\Vert \mathcal{Q} (u)-\mathcal{Q} (v)\Vert_{L^2(L^2)}\lesssim (\Vert u\Vert_{L^{\infty}(L^\infty)}+\Vert v\Vert_{L^{\infty}(L^\infty)}) \Vert u-v\Vert_{L^2(L^2)},
\end{equation}
and such that $\mathcal{Q}(0)=0$. Additionally,  
\begin{equation}\label{Lipschitz_Assumption_2}
\|\partial_t[\mathcal{Q}(u)-\mathcal{Q}(v)]\|_{L^2(L^2)}\lesssim \Vert u\Vert_{L^2(L^\infty)} \Vert u_t-v_t\Vert_{L^\infty(L^2)}+\Vert v_{t}\Vert_{L^\infty (L^2)} \Vert u-v\Vert_{L^2(L^\infty)}.  
\end{equation} 
\end{assumption}   
\noindent Note that by plugging in $v=0$ above, these assumptions further imply that
\begin{equation} \label{Q_bounds}
\begin{aligned}
\Vert \mathcal{Q} (u)\Vert_{L^2(L^2)}\lesssim&\, \Vert u\Vert_{L^{\infty}(L^\infty)} \Vert u\Vert_{L^2(L^2)}, \\
\|\partial_t [\mathcal{Q}(u)]\|_{L^2(L^2)}\lesssim&\, \Vert u\Vert_{L^2(L^\infty)} \Vert u_t\Vert_{L^\infty(L^2)}. 
\end{aligned}
\end{equation}
\noindent In~\cite{norton2016westervelt, pierce2019acoustics}, the absorption term is modeled as
\begin{align} \label{absorbed_energy_1}
\mathcal{Q}(p_t)= \frac{2 b}{\rhoa c^4_{\textup{a}}}p_t^2,
\end{align}
which clearly satisfies our assumptions if $p_t \in L^\infty(0,T; L^\infty(\Omega))$ and $p_{tt} \in L^\infty(0,T; L^2(\Omega))$. More commonly, the absorption term appears in the literature averaged over a certain time interval. In, e.g,~\cite[\S 2.2]{connor2002bio}, the absorbed energy is given by
\begin{align} \label{absorbed_energy_2}
\mathcal{Q}(p_t) = \frac{1}{j \tau} \frac{2b}{\rhoa c_{\textup{a}}^4} \int_{t'}^{t'+j \tau} p_t^2 \dt.
\end{align}
Here $j$ is a positive integer, $\tau$ is the period of ultrasound excitation and $t'$ is a sufficient time from the start of the simulation so that a steady-state has been reached. In~\cite{hallaj1999fdtd}, the absorbed energy is averaged over the whole time interval
\begin{align} \label{absorbed_energy_3}
\mathcal{Q}(p_t) = \frac{1}{T} \frac{2b}{\rhoa c_{\textup{a}}^4} \int_{0}^{T} p_t^2 \dt.
\end{align}
Both of these functionals satisfy Assumption~\ref{Q_assumption}. In case of \eqref{absorbed_energy_3}, for example,  we note that for all $t\in [0,T]$, and by using Minkowski's inequality (see \cite[Proposition 1.3]{bahouri2011fourier}), 
\begin{equation} \label{absorbed_energy_3_}
\begin{aligned}
\left\Vert \left\Vert \frac{1}{T}\int_0^T (u_t^2-v_t^2) \dt  \right\Vert_{L^2(\Omega)}\right\Vert_{L^2(0,t)}
\leq&\, \left\Vert\frac{1}{T} \int_0^T  \Vert u_t^2-v_t^2\Vert_{L^2(\Omega)}\dt\right\Vert_{L^2(0,t)}\\
=&\,  \left\Vert\frac{1}{T} \int_0^T  (\|u_t\|_{L^\infty}+\|v_t\|_{L^\infty})\|u_t-v_t\|_{L^2}\dt\right\Vert_{L^2(0,t)} \\  
\lesssim&\,  (\|u_t\|_{L^\infty(L^\infty)}+\|v_t\|_{L^\infty(L^\infty)})\|u_t-v_t\|_{L^2(L^2)}.
\end{aligned}
\end{equation}
In case of a time-averaged absorbed energy, we have $	\Vert \partial_t[\mathcal{Q}(u_t)-\mathcal{Q}(v_t)]\Vert_{L^2(L^2)}=0.$
\subsection*{Auxiliary results} 
We collect here several useful inequalities that are repeatedly used in the analysis below. We assume throughout that $\Omega \subset \R^d$, where $d \in \{1,2 ,3\}$, is an open, bounded, and sufficiently smooth set. We will often rely on the Ladyzhenskaya inequality for $u \in H^1(\Omega)$:
\begin{equation}  \label{Ladyz_Ineq}
\Vert u \Vert_{L^4}\leq 
C\Vert u \Vert_{L^2}^{1-d/4}\Vert  u\Vert_{H^1}^{d/4}. 
\end{equation} 
By using \eqref{Ladyz_Ineq} together with Young's inequality, we further find that for $u \in H_0^1(\Omega)$ and any $\varepsilon>0$ 
\begin{equation}\label{Ehrlings_ineq}
\begin{aligned}
\|u\|^2_{L^4}\lesssim \,\Vert u\Vert_{L^2}^{2(1-d/4)}\Vert u\Vert_{H^1}^{d/2}
\lesssim &\, \Vert u\Vert_{L^2}^{2(1-d/4)}\Vert \nabla u\Vert_{L^2}^{d/2}\\
\lesssim &\,
\,\frac{1}{\tilde{\varepsilon}^{\frac{4}{4-d}}}\Vert u\Vert_{L^2}^{2} +\tilde{\varepsilon}^{4/d} \Vert \nabla u\Vert_{L^2}^{2} \\
=&\,
\,C(\varepsilon)\Vert u\Vert_{L^2}^{2} +\varepsilon \Vert \nabla u\Vert_{L^2}^{2}\end{aligned}
\end{equation}
with $\varepsilon=C\tilde{\varepsilon}^{4/d}$. This estimate can also be obtained (on bounded domains) by employing Ehrling's lemma; see~\cite[Lemma 8.2]{robinson2001infinite}.\\
\indent Further, given $a \in H^{-1}(\Omega)$ and $b \in W^{1,3}(\Omega) \cap L^\infty(\Omega)$, the following bound holds:
\begin{equation} \label{Hneg_estimate}
\begin{aligned}
\|ab\|_{H^{-1}} \lesssim&\, \|a\|_{H^{-1}}(\|\nabla b\|_{L^3}+\|b\|_{L^\infty}).
\end{aligned}
\end{equation}
To keep the presentation self-contained, we also state here the version of Gronwall's inequality that will be employed in the proofs.
\begin{lemma}\label{Lemma_Gronwall}
Let $I=[0, t]$ and let $a: I\rightarrow \R$ and $b:I\rightarrow \R$ be locally integrable  functions. Let $v$ be non-negative and integrable. Suppose that $u: I\rightarrow \R$ is in $C^{1}(I)$  and satisfies: 
\begin{equation}
u^\prime(t)+v(t)\leq a(t)u(t)+b(t),\quad \text{for $t\in I$}\quad \text{and}\quad u(0)=u_0. 
\end{equation}
Then it holds that 
\begin{equation}
u(t)+ \int_{0}^t v(s)\ds\leq u_0e^{A(t)}+\int_{0}^t b(s)e^{A(t)-A(s)}\ds,
\end{equation}
where
\begin{equation}
A(t)= \int_{0}^t a(s)\ds.
\end{equation}
\end{lemma}
\begin{proof}
	The inequality follows by combining the arguments of~\cite[Appendix B]{Carl_Heikkila_2000} and \cite[Lemma 3.1]{garcke2017well}. 
\end{proof}	

	\section{Analysis of a linearized problem} \label{Sec:LinProblem}
We first analyze a de-coupled linearization of \eqref{coupled_problem_eq}, given by
	\begin{equation} \label{coupled_problem_linearized}
	\left\{ \begin{aligned}
	& \alpha(x,t)p_{tt}-r(x,t)\Delta p - b \Delta p_t= f_1(x,t), \quad &&\text{in} \ \Omega \times (0,T), \\[1mm]
	& 	\rhoa \Ca\Theta_t -\kappaa\Delta \Theta+ \rhob \Cb W(\Theta-\Thetaa) = \mathcal{Q}(p_t)+f_2(x,t), \quad &&\text{in} \ \Omega \times (0,T),
	\end{aligned} \right.
	\end{equation}
and supplemented by the boundary \eqref{coupled_problem_BC} and initial \eqref{coupled_problem_IC} conditions.
To facilitate the analysis, we make the following regularity and non-degeneracy assumptions on the involved coefficients and source terms.  
\begin{assumption}\label{Assumptions:LinProblem}
Given $T>0$, the variable coefficients and the source terms satisfy the following assumptions.
\begin{itemize}    
	\item[\bf (A)] Let $\alpha \in L^\infty(0,T; L^\infty(\Omega)\,\cap\, W^{1,3}(\Omega))$ and  $\alpha_t \in L^2(0,T; L^3(\Omega))\, \cap \, L^{\infty}(0,T; L^2(\Omega))$. 
	 Further, we assume that there exist $\alpha_0$, $\alpha_1>0$, such that  
	\[
	\alpha_0 \leq \alpha(x,t) \leq \alpha_1 \quad \text{a.e.\ in} \ \Omega \times (0,T).    
	\]
	\item[\bf (R)]  We assume that $r \in L^\infty(0,T; L^\infty(\Omega)\, \cap\, W^{1,4}(\Omega)) $ and $r_t \in L^\infty(0,T; L^{2}(\Omega))$. Further, there exist $r_0$, $r_1>0$, such that
	\[
	r_0 \leq r(x,t) \leq r_1 \quad \text{a.e.\ in} \ \Omega \times (0,T).
	\]
	\item[\bf (F)]  Let $f_1 \in L^2(0,T; H_0^1(\Om))$, $\partial_t f_1 \in L^2(0,T; H^{-1}(\Omega))$, and $f_2 \in H^1(0,T; L^2(\Omega))$.
\end{itemize}
\end{assumption}
\noindent From the last assumption, by~\cite[Theorem 7.22]{salsa2016partial}, we have $f_1 \in C([0,T]; L^2(\Omega))$ and
\begin{equation} \label{est_f1_embed}
\max_{0 \leq t \leq T} \|f_1(t)\|_{L^2} \leq C_T(\|f_1\|_{L^2(H^1)}+\|\partial_t f_1\|_{L^2(H^{-1})}).
\end{equation}
\subsection*{Energies} To accommodate the energy analysis, we introduce the following lower and higher-order acoustic energies:
\begin{equation} \label{acoustic_energies}
\begin{aligned}  
E_0[p](t)=&\, \frac{1}{2}\left\{\|\sqrt{\alpha(t)} p_t(t)\|^2_{L^2}+\|\sqrt{r}(t) \nabla p(t)\|^2_{L^2}\right\}, \\
E_1[p](t)=&\, \frac{1}{2}\left \{\|\sqrt{\alpha(t)}p_{tt}(t)\|^2_{L^2}+\|\sqrt{r}(t) \nabla p_t(t)\|^2_{L^2}+\|\sqrt{r(t)}\Delta p(t)\|^2_{L^2}\right\},\\  
E_2[p](t)=&\,\frac12  \|\sqrt{b}\nabla \Delta p(t) \|_{L^2}^2.   
\end{aligned}
\end{equation}
In the analysis, we will also use the combined acoustic energy
\begin{equation}
\mathcal{E}[p](t)=E_0[p](t)+E_1[p](t)+E_2[p](t), \qquad t \in [0,T]
\end{equation}
with the associated dissipation rate 
\begin{equation}
\begin{aligned}  
\mathcal{D}[p](t)=&\, \begin{multlined}[t]\|\sqrt{b}\nabla p_{tt}(t)\|^2_{L^2}+\|\sqrt{b}\Delta p_t(t)\|^2_{L^2}+\|\sqrt{r}(t)\nabla \Delta p(t)\|^2_{L^2}
+\|\sqrt{b}\nabla p_t(t)\|^2_{L^2}. \end{multlined}
\end{aligned}  
\end{equation}
The initial acoustic energy is set to
\begin{equation}
\begin{aligned}
\mathcal{E}[p](0)=&\, \begin{multlined}[t]\frac{1}{2}\left \{\|\sqrt{\alpha(0)}p_1\|^2_{L^2}+\|\sqrt{r(0)}\nabla p_0\|^2_{L^2}+{ \|\sqrt{r(0)}\nabla p_1\|^2_{L^2}}\right.
\left.
\right. \\
\left. +  \|\sqrt{\alpha(0)}p_{tt}(0)\|^2_{L^2} +\|\sqrt{b}\Delta\nabla  p_0\|^2_{L^2} +\|\sqrt{r(0)}\Delta p_0\|_{L^2}^2 \right\}  \end{multlined}
\end{aligned}  
\end{equation}
with
\[
p_{tt}(0)=\alpha(0)^{-1}(r(0)\Delta p_0+b \Delta p_1 +f_1(0)).
\]

\noindent Further, the heat energy is given by
\begin{equation}
\begin{aligned}
\mathcal{E}[\Theta](t)=&\, \frac{1}{2}\left \{\|\Theta(t)\|_{H^2}^2+\|\Theta_t(t)\|^2_{L^2}  \right\} 
\end{aligned}
\end{equation}
with the associated dissipation
\begin{equation}
\begin{aligned}
\mathcal{D}[\Theta](t)=&\,  \|\Theta_t(t)\|^2_{H^1}+\|\Theta_{tt}(t)\|^2_{H^{-1}}. 
\end{aligned}
\end{equation}
\subsection*{Solution spaces} To formulate the existence result, we also introduce the following solutions spaces for the pressure:
\begin{equation} \label{space_Xp}
\begin{aligned}
X_{p}= \{p \in L^\infty(0,T; \Honethree):&\  p_t \in  L^\infty(0,T; \Honetwo)\, \cap \, L^2(0,T; \Honethree), \\
&\, p_{tt} \in L^\infty(0,T; L^2(\Omega)) \cap L^2(0,T; H_0^1(\Om)), \\
&\, p_{ttt}\in L^2(0,T; H^{-1}(\Omega))\},
\end{aligned}
\end{equation}  
and the temperature:
\begin{equation} \label{space_XTheta}
\begin{aligned}
X_{\Theta}=\{\Theta \in C([0,T]; \Honetwo):& \, \Theta_t \in C([0,T]; L^2(\Om)) \cap L^2(0,T; H^1(\Omega)), \\ &\,\Theta_{tt}\in L^2(0,T; H^{-1}(\Omega))\},
\end{aligned}
\end{equation}
with the short-hand notation 
\begin{equation} \label{sobolev_withtraces}
\begin{aligned}
\Honetwo=&\,H_0^1(\Omega)\cap H^2(\Omega),\\
\Honethree=&\, \left\{u\in H^3(\Omega)\,:\, \mbox{tr}_{\partial\Omega} u = 0, \  \mbox{tr}_{\partial\Omega} \D u = 0\right\}.
\end{aligned}
\end{equation}
We claim that the linearized problem is well-posed under the above-made assumptions.
\begin{proposition} \label{Prop:LinWellp}
Let $T>0$ and let Assumption~\ref{Assumptions:LinProblem} hold. Further, assume that
\begin{equation}    
\begin{aligned}  
(p_0, p_1) \in \Honethree \times \Honetwo, \quad \Theta_0 \in \Honetwo.
\end{aligned}
\end{equation}
Then there exists a unique solution $(p, \Theta) \in X_p \times X_{\Theta}$ of \eqref{coupled_problem_linearized}. Furthermore, the acoustic pressure satisfies 
\begin{equation}\label{Main_Energy_Estimate_0}
\begin{aligned}
&\mathcal{E}[p](t)+\|\Delta p_t(t)\|^2_{L^2}  +\int_0^t\mathcal{D}[p](s)\ds+\int_0^t(\|p_{ttt}(s)\|^2_{H^{-1}}+\|\nabla \Delta p_t(s)\|^2_{L^2})\ds\\
\lesssim &\, \mathcal{E}[p](0)\exp{\left(\int_0^t (1+\Lambda(s))\ds\right)}+\int_0^t\exp{\left(\int_s^t (1+\Lambda(\sigma))\textup{d}\sigma\right)}\mathbb{F}(s)\ds
\end{aligned}
\end{equation}
a.e.\ in time, with    
\begin{equation} \label{def_Lamba}
\Lambda(t)=\left\Vert r_t(t)\right\Vert_{L^2}^2+\|\nabla r(t)\|_{L^4}+\left\Vert \alpha_t(t)\right\Vert_{L^2}+\Vert \alpha_t(t)\Vert_{L^3}^2+\left\|\nabla \alpha(t)\right\|_{L^3}^2 
\end{equation}  
and    
\begin{equation}\label{F_Terms}
\mathbb{F}(t)=\Vert f_1(t)\Vert_{H^1}^2+(1+\left\|\nabla \alpha(t)\right\|_{L^3}^2)\Vert \partial_tf_1(t)\Vert_{H^{-1}}^2,
\end{equation}
whereas the temperature satisfies
\begin{equation}
\begin{aligned}        
&\mathcal{E}[\Theta](t)+\int_0^t\mathcal{D}[\Theta](s)\ds\\
\leq&\, C_T \big(\|\Theta_0\|_{\Honetwo}^2+\|f_2\|^2_{H^1(L^2)}+ \|p_t\|^2_{L^\infty(L^\infty)}\|p_t\|^2_{L^2(L^2)}+ \|p_t\|_{L^2(L^\infty)}^2\| p_{tt}\|^2_{L^\infty(L^2)}+1\big)
\end{aligned}
\end{equation}
for all $t \in [0,T]$. 
\end{proposition}
\begin{proof}
	Since the system is de-coupled, we can analyze the equations in \eqref{coupled_problem_linearized}  sequentially.
\subsection*{Analysis of the pressure equation} The analysis of the pressure equation can be rigorously conducted by employing a Galerkin discretization in space based on the smooth eigenfunctions of the Dirichlet-Laplacian; see. e.g.,~\cite[Ch.\ 7]{evans2010partial}. We focus here on presenting the energy analysis. 
\subsection*{Energy analysis}\label{Energy analysis}\label{Sec_Energy_Analysis}
Testing the (semi-discrete) pressure equation with $p_t$, integrating over $\Omega$, and using integration by parts yields the following identity:
\begin{equation}
\begin{aligned}
\frac12 \ddt \|\sqrt{\alpha(t)} p_t(t)\|^2_{L^2} +\|\sqrt{b}\nabla p_t(t)\|^2_{L^2}=&\,\frac12 (\alpha_t p_t, p_t)_{L^2}+( r  \Delta p, p_t)_{L^2}+(f_1, p_t)_{L^2}
\end{aligned}  
\end{equation}
a.e.\ in time. From here, by H\"older's and Young's inequalities, we have
\begin{equation}
\begin{aligned}
&\frac12 \ddt \|\sqrt{\alpha(t)} p_t(t)\|^2_{L^2}+\|\sqrt{b}\nabla p_t(t)\|^2_{L^2}\\
\lesssim&\, \begin{multlined}[t]\left\Vert \frac{\alpha_t(t)}{b}\right\Vert_{L^2}\|\sqrt{b} p_t(t)\|^2_{L^4}  
+\left\Vert\sqrt{\frac{r(t)}{\alpha(t)}} \right\Vert_{L^\infty}(\|\sqrt{r(t)}\Delta p(t)\|^2_{L^2}+\|\sqrt{\alpha(t)} p_t(t)\|^2_{L^2})\\
+\frac{1}{\sqrt{b}} \Vert f_1(t)\Vert_{L^2}\|\sqrt{b} p_t(t)\|_{L^2}. \end{multlined}
\end{aligned}
\end{equation}
On account of Assumption \ref{Assumptions:LinProblem}, we know that 
\[
\left\Vert\sqrt{{r(t)}/{\alpha(t)}} \right\Vert_{L^\infty} \leq \sqrt{{r_1}/{\alpha_0}} \quad \text{a.e.\ in time},
\]
and thus for any $\varepsilon>0$, it holds that 
\begin{equation}\label{E_0_Estimate}
\begin{aligned}
&\frac12 \ddt \|\sqrt{\alpha(t)} p_t(t)\|^2_{L^2}+\|\sqrt{b}\nabla p_t(t)\|^2_{L^2}\\
\lesssim&\, \begin{multlined}[t]\left\Vert \frac{\alpha_t(t)}{b}\right\Vert_{L^2}\|\sqrt{b} p_t(t)\|^2_{L^4}
+ E_0[p](t)+ E_1[p](t)+\frac{1}{4\varepsilon}\Vert f_1(t)\Vert_{L^2}^2+\varepsilon\|\sqrt{b} \nabla p_t(t)\|_{L^2}^2, \end{multlined}
\end{aligned}
\end{equation}
where we have applied Poincare's inequality together with Young's $\varepsilon$-inequality in the estimate of the last term.  Note that by fixing $\varepsilon>0$ small enough, we can absorb the last term in \eqref{E_0_Estimate} by the dissipative term on the left. \\ 
\indent By  using the embedding $H^1(\Om)\hookrightarrow L^4(\Om)$ together with the Poincar\'e inequality,  the first term on the right-hand side of \eqref{E_0_Estimate} can be absorbed by the dissipative term $\|\sqrt{b}\nabla p_t(t)\|^2_{L^2}$ as well if we assume  the norm $\left\Vert \alpha_t/b\right\Vert_{L^\infty(L^2)}$ to be small. However, to avoid this smallness assumption, we use inequality \eqref{Ehrlings_ineq} instead and split this term into two parts: an energy term and a dissipation term with an arbitrary small factor $\varepsilon>0$. This idea will be used repeatedly in the proof below. Indeed, by using inequality \eqref{Ehrlings_ineq}, we have
\begin{equation}\label{L_4_Term_E_0}
\begin{aligned}
\|\sqrt{b} p_{t}(t)\|^2_{L^4}
\lesssim&\,
C(\varepsilon)\Big\Vert \frac{b}{\alpha(t)}\Big\Vert_{L^\infty} \Vert \sqrt{\alpha} p_{t}(t)\Vert_{L^2}^{2} +\varepsilon \Vert \sqrt{b} \nabla p_{t}(t)\Vert_{L^2}^{2}.
\end{aligned}
\end{equation}
Consequently, by recalling Assumption \ref{Assumptions:LinProblem} and fixing $\varepsilon>0$ small enough, so that 
\[1-C \varepsilon \displaystyle \sup_{t \in (0,T)}\left\Vert \alpha_t(t)/b\right\Vert_{L^2}>0,\] where $C$ is the hidden constant in \eqref{E_0_Estimate}, we obtain 
\begin{equation}\label{E_0_Estimate_2}
\begin{aligned}
&\ddt E_0[p](t)+\|\sqrt{b}\nabla p_t(t)\|^2_{L^2}\\
\lesssim&\,\begin{multlined}[t]  E_0[p](t)+E_1[p](t)+\left\Vert \frac{\alpha_t(t)}{b}\right\Vert_{L^2}\|\sqrt{\alpha(t)}  p_t(t)\|^2_{L^2}
+\Vert f_1(t)\Vert_{L^2}^2,\end{multlined}
\end{aligned}
\end{equation}
where we have also used again the uniform bound on $\alpha$ given in Assumption \ref{Assumptions:LinProblem}.\\
\indent Estimate \eqref{E_0_Estimate_2} indicates that further testing is needed to absorb the energy $E_1$ on the right. Thus, we test the first (semi-discrete) equation in \eqref{coupled_problem_linearized} with $-\Delta p_t$ and integrate in space, which yields 
\begin{equation}\label{Energy_2_Estimate}
\begin{aligned}    
&\frac12 \ddt \|\sqrt{r(t)}\Delta p(t)\|^2_{L^2}+\|\sqrt{b}\Delta p_t(t)\|^2_{L^2}\\
 =&\, (\alpha(t) p_{tt}, \Delta p_t)_{L^2}+\frac12 (r_t(t) \Delta p, \Delta p)-(f_1(t), \Delta p_t)_{L^2}\\
\lesssim&\, \begin{multlined}[t]\frac{1}{4\varepsilon}\|\sqrt{\alpha(t)}p_{tt}(t)\|^2_{L^2}+\varepsilon\left\Vert \sqrt{\frac{\alpha(t)}{b}}\right\Vert_{L^\infty} \|\sqrt{b}\Delta p_t(t)\|^2_{L^2}
+\left\Vert \frac{r_t(t)}{b}\right\Vert_{L^2}\|\sqrt{b}\Delta p(t)\|^2_{L^4}\\
+\frac{1}{b}  \Vert f_1(t)\Vert_{L^2}^2+\varepsilon\|\sqrt{b} \Delta p_t(t)\|_{L^2}^2. \end{multlined}  
\end{aligned}  
\end{equation}
Clearly, by selecting $\varepsilon>0$ small enough in the above estimate, the second term on the right-hand side will be absorbed by the dissipation on the left. Hence, by choosing $\varepsilon>0$ as small as needed, keeping in mind that $\Delta p=0$ on $\partial \Omega$, and using Poincar\'e's inequality,  we obtain  
\begin{equation}\label{Energy_Delta_p}
\begin{aligned}    
 &\frac12 \ddt \|\sqrt{r}\Delta p(t)\|^2_{L^2}+\|\sqrt{b}\Delta p_t(t)\|^2_{L^2}\\  
 \lesssim&\,E_1[p](t)+\left\Vert \frac{r_t(t)}{b}\right\Vert_{L^2}\|\sqrt{b}\nabla\Delta p(t)\|^2_{L^2}
 +\Vert f_1(t)\Vert_{L^2}^2. 
 \end{aligned} 
\end{equation}
To retrieve the energy $E_1$ on the left, we will next work with the time-differentiated pressure equation. Indeed, on account of the regularity assumptions on the coefficients and source term, we can differentiate the semi-discrete pressure equation with respect to $t$:
\begin{equation}\label{Eq:p_ttt}
\alpha(x,t)p_{ttt}-r(x,t)\Delta p_t - b \Delta p_{tt}= \partial_tf_1(x,t)-\alpha_t(x,t)p_{tt}+r_t(x,t)\Delta p.
\end{equation}
Multiplying \eqref{Eq:p_ttt} by $p_{tt}$, integrating over $\Omega$, and using integration by parts with respect to time in the first term, we obtain 
\begin{equation}\label{deriva_Energy_p_tt}
\begin{aligned}
&\frac{1}{2}\ddt\left\{\|\sqrt{\alpha(t)}p_{tt}(t)\|^2_{L^2}+\Vert \sqrt{r}(t)\nabla p_t(t)\Vert_{L^2}^2\right\}+\|\sqrt{b}\nabla p_{tt}(t)\|^2_{L^2}\\
=&\, \begin{multlined}[t]\frac12(\alpha_t p_{tt}, p_{tt})_{L^2}-(\nabla r  p_t, \nabla p_{tt})_{L^2}+\frac{1}{2}(r_t \nabla p_t, \nabla p_t)_{L^2}+\langle\partial_t f_{1}, p_{tt}\rangle_{H^{-1}, H^1}\\-(\alpha_t p_{tt},p_{tt})_{L^2}+(r_t \Delta p, p_{tt})_{L^2}. \end{multlined}
\end{aligned}
\end{equation}
The first two $r$ terms on the right can be estimated as follows:
\begin{equation}\label{New_Estimate_E_1_1}
\begin{aligned}
&-(\nabla r  p_t, \nabla p_{tt})_{L^2}+\frac{1}{2}(r_t \nabla p_t, \nabla p_t)_{L^2} \\
\leq&\,\begin{multlined}[t]\varepsilon\|\sqrt{b}\nabla p_{tt}(t)\|^2_{L^2}+C(\varepsilon)\left\Vert \frac{1}{\sqrt{r}}\right\Vert_{L^\infty} \Vert \nabla r\Vert_{L^4}\Vert \sqrt{r}\nabla p_t\Vert_{L^2} 
+\frac12(r_t \nabla p_t, \nabla p_t)_{L^2}\end{multlined}
 \end{aligned}
\end{equation}
for some $\varepsilon>0$, where we have relied on the embedding $H^1(\Omega)\hookrightarrow L^4 (\Omega)$. By applying estimate \eqref{Ehrlings_ineq}, we can further bound the last term: 
\begin{equation}\label{Last_Term_E_1} 
\begin{aligned}
\frac12 (r_t \nabla p_t, \nabla p_t)_{L^2} \lesssim&\,  \Vert r_t\Vert_{L^2}\Vert \nabla p_t\Vert_{L^4}^2\\
\lesssim&\,  \,
C(\varepsilon)\Vert r_t\Vert_{L^2}^2\| r^{-1}\|_{L^\infty}\Vert \sqrt{r}\nabla p_t\Vert_{L^2}^{2} +\varepsilon \Vert \Delta  p_t\Vert_{L^2}^{2},
 \end{aligned}
\end{equation}
where we have also utilized elliptic regularity (since $\partial \Omega$ is smooth):
\[
\|\nabla p_t\|_{H^1} \leq \|p_t\|_{H^2} \leq C \|\Delta p_t\|_{L^2}.
\]
The first and the fifth term on the right-hand side of \eqref{deriva_Energy_p_tt} can be estimated as follows:  
\begin{equation}\label{Energ_3_First_term}
\begin{aligned}
\frac12(\alpha_t p_{tt}, p_{tt})_{L^2}-(\alpha_t p_{tt},p_{tt})_{L^2}=\,-\frac12(\alpha_t p_{tt},p_{tt})_{L^2} 
\lesssim&\,\left\Vert \frac{\alpha_t(t)}{b}\right\Vert_{L^2}\|\sqrt{b} p_{tt}(t)\|^2_{L^4}.
\end{aligned}
\end{equation}
We then further estimate the last term above using again inequality \eqref{Ehrlings_ineq}:
\begin{equation}\label{Ladyz_Estimate_1}
\begin{aligned}  
\|\sqrt{b} p_{tt}(t)\|^2_{L^4} \lesssim&\,
C(\varepsilon)\Big\Vert \frac{b}{\alpha(t)}\Big\Vert_{L^\infty} \Vert \sqrt{\alpha(t)} p_{tt}(t)\Vert_{L^2}^{2} +\varepsilon \Vert \sqrt{b} \nabla p_{tt}(t)\Vert_{L^2}^{2}.
\end{aligned}
\end{equation}
Keeping in mind Assumption \eqref{Assumptions:LinProblem}, and plugging \eqref{Ladyz_Estimate_1} into \eqref{Energ_3_First_term}, we have 
\begin{equation}\label{Energ_3_First_term_2}
\begin{aligned}  
&-\frac12(\alpha_t p_{tt}, p_{tt})_{L^2}
\lesssim\,\left\Vert \frac{\alpha_t(t)}{b}\right\Vert_{L^2}\Big(\varepsilon\Vert \sqrt{b}\nabla  p_{tt}(t)\Vert_{L^2}^{2}+C(\varepsilon) \Vert \sqrt{\alpha} p_{tt}(t)\Vert_{L^2}^{2}\Big).
\end{aligned}
\end{equation}
By using Young's inequality together with the Poincar\'e's inequality, we find that
\begin{equation}\label{Energ_pttt}
\begin{aligned}
\langle \partial_t f_{1}(t), p_{tt}(t)\rangle_{H^{-1}, H^{1}}\lesssim &\,\frac{1}{\sqrt{b}} \Vert \partial_tf_1(t)\Vert_{H^{-1}}\|\sqrt{b}  p_{tt}(t)\|_{H^{1}}\\
\lesssim&\, 4\varepsilon \frac{1}{b} \Vert \partial_tf_1(t)\Vert_{H^{-1}}^2+\varepsilon\|\sqrt{b}\nabla   p_{tt}(t)\|_{L^2}^2.
\end{aligned}
\end{equation}
Recalling that $\Delta p=0$ on $\partial \Omega$, we can estimate the last term on the right-hand side of \eqref{deriva_Energy_p_tt} as follows:
\begin{equation}\label{Last_Term_E_1}
\begin{aligned}
(r_t \Delta p, p_{tt})_{L^2}\lesssim&\, \varepsilon \Vert \sqrt{b}p_{tt}\Vert_{L^4}^2+C(\varepsilon)\left\Vert  \frac{r_t}{b}\right\Vert_{L^2}^2\Vert \sqrt{b}\Delta p\Vert_{L^4}^2   \\
\lesssim&\, \varepsilon \Vert \sqrt{b}\nabla p_{tt}\Vert_{L^2}^2+C(\varepsilon)\left\Vert  \frac{r_t}{b}\right\Vert_{L^2}^2\Vert \sqrt{b}\nabla\Delta p\Vert_{L^2}^2.
\end{aligned}  
\end{equation} 
We see that the first term on the right can be absorbed by the dissipation in \eqref{deriva_Energy_p_tt} and the last one is an energy term. By collecting the above estimates with  $\varepsilon>0$ small enough, we arrive at
\begin{equation}\label{deriva_Energy_p_tt_2}
\begin{aligned}
&\frac{1}{2}\ddt\Big(\|\sqrt{\alpha}p_{tt}(t)\|^2_{L^2}+\Vert \sqrt{r}(t)\nabla p_t(t)\Vert_{L^2}^2\Big)+\|\sqrt{b}\nabla p_{tt}(t)\|^2_{L^2}\\
\lesssim &\, \begin{multlined}[t]\left\Vert  \frac{r_t}{b}\right\Vert_{L^2}^2\Vert \sqrt{b}\nabla\Delta p\Vert_{L^2}^2+(\Vert \nabla r\Vert_{L^4}+\Vert r_t\Vert_{L^2}^2)\Vert \sqrt{r}\nabla p_t\Vert_{L^2}\\+\Vert \partial_tf_1(t)\Vert_{H^{-1}}^2
+\varepsilon \Vert \Delta  p_t\Vert_{L^2}^{2}.
\end{multlined}
\end{aligned}
\end{equation}
Adding \eqref{deriva_Energy_p_tt_2} to \eqref{Energy_Delta_p}, exploiting Assumption \ref{Assumptions:LinProblem}, using Poincar\'e's inequality, and possibly reducing $\varepsilon$, so that the $\varepsilon$ terms can be absorbed by the left side, we obtain  
\begin{equation}\label{Energy_Delta_p_p_tt}
\begin{aligned}  
 &\begin{multlined}[t] \ddt \underbrace{\frac12\Big[\|\sqrt{r(t)}\Delta p(t)\|^2_{L^2}+\|\sqrt{\alpha(t)}p_{tt}(t)\|^2_{L^2}{+\Vert \sqrt{r}(t)\nabla p_t(t)\Vert_{L^2}^2}\Big]}_{:=E_1[p](t)}\\+\|\sqrt{b}\nabla p_{tt}(t)\|^2_{L^2}+\|\sqrt{b}\Delta p_t(t)\|^2_{L^2}\end{multlined}\\  
 \lesssim &\,  \begin{multlined}[t] \left(1+\Vert \nabla r\Vert_{L^4}+\Vert r_t\Vert_{L^2}^2\right)E_1[p](t)+\left\Vert \frac{r_t(t)}{b}\right\Vert_{L^2}^2\|\sqrt{b} \nabla \Delta p(t)\|^2_{L^2}
 +\Vert f_1(t)\Vert_{L^2}^2\\
 +\Vert \partial_tf_1(t)\Vert_{H^{-1}}^2.
\end{multlined}
 \end{aligned} 
\end{equation}
To be able to absorb the term $\|\sqrt{b} \nabla \Delta p(t)\|^2_{L^2}$ on the right, we should additionally test the pressure equation with $\Delta^2 p$:
\[
(\alpha(t)p_{tt}-r(t)\Delta p - b \Delta p_t, \Delta^2 p)_{L^2}= (f_1(t), \Delta^2 p)_{L^2}. 
\]
Integrating  by parts and using the fact that $p_{tt}=\Delta p= \Delta p_t=0$ on the boundary for smooth Galerkin approximations, as well as that $f_1(t) \in H_0^1(\Om)$, yields
\[
(r \nabla \Delta p +b \nabla \Delta p_t, \nabla \Delta p)_{L^2}= -(\alpha \nabla p_{tt}+p_{tt}\nabla \alpha+ \nabla r\Delta p, \nabla \Delta p)_{L^2}+(\nabla f_1, \nabla \Delta p)_{L^2}.
\]
Recalling how the energy $E_2$ is defined in \eqref{acoustic_energies}, from here we obtain 
\begin{equation}
\begin{aligned}
&\,\ddt E_2[p](t)+\|\sqrt{r}(t)\nabla \Delta p(t)\|^2_{L^2}\\
=&\,-(\alpha \nabla p_{tt}+p_{tt}\nabla \alpha+ \nabla r\Delta p, \nabla \Delta p)_{L^2}+(\nabla f_1(x,t), \nabla \Delta p)_{L^2}.
\end{aligned}
\end{equation}
By H\"older's inequality, we further have
\[
\begin{aligned}
& \ddt E_2[p](t)+\|\sqrt{r}(t)\nabla \Delta p(t)\|^2_{L^2}\\
\lesssim&\,\begin{multlined}[t] \|\alpha(t)\|_{L^\infty}\|\nabla p_{tt}(t)\|_{L^2}\|\nabla \Delta p(t)\|_{L^2}+{\|p_{tt}(t)\|_{L^6}\|\nabla \alpha(t)\|_{L^3}\|\nabla \Delta p(t)\|_{L^2}}\\+\|\nabla r(t)\|_{L^4}\|\Delta p(t)\|_{L^4}\|\nabla \D p(t)\|_{L^2}+\frac{1}{4\varepsilon}\|\nabla f_1(t)\|_{L^2}^2+\varepsilon\Vert r(t)^{-1}\Vert_{L^\infty} \|\sqrt{r(t)}\nabla \Delta p(t)\|_{L^2}. \end{multlined}
\end{aligned}
\]
Using Young's and Poincar\'e's inequalities, and the embedding $H^1(\Omega) \hookrightarrow L^6(\Omega)$ yields
\begin{equation}\label{Eq_E_2_dt}
\begin{aligned}
& \ddt E_2[p](t)+\|\sqrt{r}(t)\nabla \Delta p(t)\|^2_{L^2}\\
\lesssim&\,\begin{multlined}[t] \frac{\varepsilon}{b}\|\alpha(t)\|^2_{L^\infty}\|\sqrt{b}\nabla p_{tt}(t)\|_{L^2}^2+\frac{1}{4 \varepsilon b}\|\sqrt{b}\nabla \Delta p(t)\|^2_{L^2}
+\varepsilon\|\nabla p_{tt}(t)\|^2_{L^2}\|\nabla \alpha(t)\|^2_{L^3}\\
+\|\sqrt{b}\nabla \Delta p(t)\|^2_{L^2}
+\|\nabla r(t)\|_{L^4}\|\sqrt{b}\nabla \D p(t)\|^2_{L^2}  
+\|\nabla f_1(t)\|_{L^2}^2. \end{multlined}
\end{aligned}
\end{equation}
By adding inequalities \eqref{Energy_Delta_p_p_tt} and \eqref{Eq_E_2_dt}, and selecting $\varepsilon>0$ small enough,  we have
\begin{equation}\label{Energy_Delta_p_Delta_p_t}
\begin{aligned}  
 & \begin{multlined}[t] \ddt \left\{E_1[p](t)+E_2[p](t)\right\}
 +\|\sqrt{b}\nabla p_{tt}(t)\|^2_{L^2}+\|\sqrt{b}\Delta p_t(t)\|^2_{L^2}\\+\|\sqrt{r(t)}\nabla \Delta p(t)\|^2_{L^2}\end{multlined} \\  
\lesssim&\,  \begin{multlined}[t]\left(1+\Vert \nabla r(t)\Vert_{L^4}+\Vert r_t\Vert_{L^2}^2\right)\left\{E_1[p](t)+E_2[p](t)\right\}
+\Vert \partial_tf_1(t)\Vert_{H^{-1}}^2+\|f_1(t)\|_{H^1}^2. \end{multlined}
 \end{aligned} 
\end{equation}
By collecting the above estimates, we arrive at a bound that involves the combined acoustic energy: 
\begin{equation}\label{E_p_cal}
\begin{aligned}
 \ddt\mathcal{E}[p](t)+\mathcal{D}[p](t)\lesssim (1+\Lambda(t))\mathcal{E}[p](t)+\mathbb{F}(t), 
\end{aligned} 
\end{equation} 
where $\Lambda(t)$ and $\mathbb{F}(t)$ are defined in \eqref{def_Lamba} and \eqref{F_Terms}, respectively. By Gronwall's inequality, we then immediately have
\begin{equation}\label{E_p_cal_Main}
\begin{aligned}
&\mathcal{E}[p](t) +\int_0^t\mathcal{D}[p](s)\ds \\
\lesssim &\, \mathcal{E}[p](0)\exp{\left(\int_0^t (1+\Lambda(s))\ds\right)}+\int_0^t\exp{\left(\int_s^t (1+\Lambda(\sigma))\textup{d}\sigma\right)}\mathbb{F}(s)\ds.
\end{aligned} 
\end{equation}
\noindent\emph{Additional bootstrap arguments.} We can obtain more information on the pressure field by relying on the (semi-discrete) PDE. Indeed, by the acoustic PDE we have
\begin{equation}\label{Delta_p_t}
\begin{aligned}
\|\Delta p_t(t)\|^2_{L^2} \lesssim \alpha_1^2 \|p_{tt}(t)\|^2_{L^2}+r_1^2\|\Delta p(t)\|_{L^2}^{{2}}
+\|f_1(t)\|_{L^2}^2. 
\end{aligned}
\end{equation} 
 We can then further estimate the right-hand side of \eqref{Delta_p_t} by employing the acoustic energy:
	\begin{equation}
	\begin{aligned}
	\|\Delta p_t(t)\|^2_{L^2} \lesssim&\, \mathcal{E}[p](t)
	+\|f_1(t)\|_{L^2}^2 \\
	\lesssim&\, \mathcal{E}[p](0)\exp{\left(\int_0^t (1+\Lambda(s))\ds\right)}+\int_0^t\exp{\left(\int_s^t (1+\Lambda(\sigma))\textup{d}\sigma\right)}\mathbb{F}(s)\ds,
	\end{aligned}
	\end{equation} 
where we have also used estimate \eqref{est_f1_embed} to bound the $\|f_1(t)\|_{L^2}^2$ term. Adding this bound to \eqref{E_p_cal_Main} yields
\begin{equation}\label{E_p_cal_Main_}
\begin{aligned}
&\mathcal{E}[p](t)+\|\Delta p_t(t)\|^2_{L^2} +\int_0^t\mathcal{D}[p](s)\ds\\
\lesssim&\, \mathcal{E}[p](0)\exp{\left(\int_0^t (1+\Lambda(s))\ds\right)}+\int_0^t\exp{\left(\int_s^t (1+\Lambda(\sigma))\textup{d}\sigma\right)}\mathbb{F}(s)\ds.
\end{aligned} 
\end{equation}
Similarly,
	\begin{equation} \label{higher_pt}
\begin{aligned}
\|\sqrt{b}\nabla \Delta p_t(t)\|^2_{L^2} \lesssim&\, \alpha_1^2 \|\nabla p_{tt}(t)\|_{L^2}^2+\|\nabla \alpha(t)\|_{L^3}\|p_{tt}(t)\|_{L^6}\\
&+r_1\|\sqrt{r}\nabla \Delta p(t)\|^2_{L^2}+\|\nabla r(t)\|^2_{L^4}\|\Delta p(t)\|^2_{L^4}
+\|\nabla f_1(t)\|_{L^2}^2. 
\end{aligned}
\end{equation} 
Adding $\gamma\cdot$\eqref{higher_pt} to \eqref{E_p_cal} with small enough $\gamma>0$ yields 
\begin{equation}
\begin{aligned}
\ddt\mathcal{E}[p](t)+\mathcal{D}[p](t)+\|\sqrt{b}\nabla \Delta p_t(t)\|^2_{L^2}\lesssim (1+\Lambda(t))\mathcal{E}[p](t)+\mathbb{F}(t),
\end{aligned} 
\end{equation}
on which we can apply Gronwall's inequality.\\
\indent Additionally, from the time-differentiated equation \eqref{Eq:p_ttt}, standard arguments (see, e.g.,~\cite[Ch.\ 7, p.\ 383]{evans2010partial}) give the following bound in the dual space $H^{-1}(\Omega)$:
\begin{equation}\label{p_ttt_1}  
\begin{aligned}
&\|\partial_t(\alpha(t) p_{tt})(t)\|_{H^{-1}}\\
 \leq&\, \|r(t)\Delta p_t(t)\|_{H^{-1}}+\|r_t(t)\Delta p(t)\|_{H^{-1}}+\|b\Delta p_{tt}(t)\|_{H^{-1}}+\|\partial_t f_1(t)\|_{H^{-1}} \\  
\lesssim&\, \begin{multlined}[t]\|r(t)\|_{L^\infty}\|\Delta p_t(t)\|_{L^2}+\|r_t(t)\|_{L^2}\|\nabla \Delta p(t)\|_{L^2}+\|\nabla p_{tt}(t)\|_{L^2}+\|\partial_t f_1(t)\|_{H^{-1}}, \end{multlined}
\end{aligned}
\end{equation}
where we have used the embedding $L^{6/5}(\Omega) \hookrightarrow H^{-1}(\Omega) $ together with H\"older's inequality to get \[\|r_t \Delta p\|_{H^{-1}} \lesssim \|r_t \Delta p\|_{L^{6/5}} \lesssim \|r_t\|_{L^2}\|\Delta p\|_{L^{3}} \lesssim \|r_t\|_{L^2}\|\nabla \Delta p\|_{L^2}.\]  
 Thus, we have
\[
p_{tt} \in L^2(0,T; H_0^1(\Omega)), \quad \partial_t (\alpha(\cdot) p_{tt}) \in L^2(0,T; H^{-1}(\Omega))
\]   
with a uniform bound
\begin{equation}\label{p_ttt_L^2}
\begin{aligned}    
	\|p_{ttt}\|_{H^{-1}} \lesssim&\, \|\alpha p_{ttt}\|_{H^{-1}}\left( \left\|\alpha^{-1}\right\|_{L^\infty}+\left\|\nabla (\alpha^{-1})\right\|_{L^3}\right)\\
	\lesssim&\, (\|\partial_t(\alpha p_{tt})\|_{H^{-1}}+\|\alpha_t p_{tt}\|_{H^{-1}})\left( \alpha_1^{-1}+\alpha_1^{-2}\left\|\nabla \alpha\right\|_{L^3}\right).
\end{aligned}   
\end{equation}
By using again the embedding $L^{6/5}(\Omega) \hookrightarrow H^{-1}(\Omega)$ and H\"older's inequality, we have, similarly to before,
 \begin{equation}\label{p_tt_H_1}
\|\alpha_t p_{tt}\|_{H^{-1}}\lesssim  \|\alpha_t p_{tt}\|_{L^{6/5}}\lesssim\Vert \alpha_t\Vert_{L^3}\Vert p_{tt}\Vert_{L^2},
\end{equation}
and thus
\begin{equation}\label{p_ttt_L^2_3}
\begin{aligned}    
\|p_{ttt}\|_{H^{-1}}^2 
\lesssim&\,\begin{multlined}[t] (1+\|\nabla \alpha\|_{L^3}^2)\left(\|r\|_{L^\infty}^2\|\Delta p_t\|_{L^2}^2 
+\|r_t\|_{L^2}^2\|r^{-1}\|_{L^\infty} \|\sqrt{r}\nabla \Delta p\|_{L^2}^2\right.\\+\|\nabla p_{tt}\|_{L^2}^2\left.
+\|\partial_t f_1\|_{H^{-1}}^2+\Vert \alpha_t\Vert_{L^3}^2\Vert p_{tt}\Vert_{L^2}^2\right).
\end{multlined}
\end{aligned}
\end{equation}
Then adding $\gamma \cdot$\eqref{p_ttt_L^2_3} to \eqref{E_p_cal} with $\gamma>0$ small enough, and  using Gronwall's inequality yields 
\begin{equation}\label{E_p_cal_Main_2}
\begin{aligned}
&\mathcal{E}[p](t)+\|p_{ttt}\|_{L^2(H^{-1})}^2 +\int_0^t\mathcal{D}[p](s)\ds\\
\lesssim&\, \mathcal{E}[p](0)\exp{\left(\int_0^t (1+\Lambda(s))\ds\right)}+\int_0^t\exp{\left(\int_s^t (1+\Lambda(\sigma))\textup{d}\sigma\right)}\mathbb{F}(s)\ds. 
\end{aligned} 
\end{equation}
Combining the three derived estimates yields \eqref{Main_Energy_Estimate_0}, as first in a semi-discrete setting. The obtained uniform bound allows us to employ standard compactness arguments and prove existence of a solution $p \in X_p$ to the pressure equation; see, e.g.,~\cite[Ch.\ 7]{evans2010partial} for similar arguments. By the weak/weak-$\star$ lower semi-continuity of norms, $p$ satisfies the same energy bound   \eqref{Main_Energy_Estimate_0}. Note that $p \in X_p$ implies
\begin{equation}
\begin{aligned}
p \in C([0,T]; \Honethree), \quad p_t \in  C_w([0,T]; \Honetwo); 
\end{aligned}
\end{equation}
cf. \cite[Lemma 3.3]{temam2012infinite}.\\[2mm]  
\noindent \emph{Uniqueness.} Uniqueness in the pressure equation follows by showing that the only solution of the homogeneous problem is zero. To this end,
let ${p} \in X_p$ solve
\begin{equation}\label{Eq_Difference}
\alpha(x,t){p}_{tt}-r(x,t)\Delta {p} - b \Delta {p}_t=0,\qquad {p}(x,0)={p}_t(x,0)=0,\qquad {p}\vert_{\partial \Om}=0. 
\end{equation}
We can repeat our previous energy analysis up to \eqref{E_p_cal}, where instead of testing with $\Delta^2 p$ (which is not a valid test function), we take the gradient of the equation and test with $\nabla \Delta p \in L^\infty(L^2(\Omega))$. In this manner, from \eqref{Main_Energy_Estimate_0} we obtain  $\mathcal{E}[{p}](t)=0$, which immediately yields ${p}=0$.

\subsection*{Analysis of the heat equation.} 
We next rewrite the heat equation as
\[
\Theta_t -\frac{\kappaa}{\rhoa \Ca}\Delta \Theta+ \frac{\rhob \Cb W}{\rhoa \Ca}\Theta =\tilde{f}
\]
with
\[
\tilde{f}=\frac{1}{\rhoa \Ca} \mathcal{Q}(p_t)+\frac{1}{\rhoa \Ca}f_2(x,t)+\frac{\rhob \Cb W \Thetaa}{\rhoa \Ca}.
\]
According to, e.g.,~\cite[Ch.\ 1, Theorem 1.3.2]{zheng2004nonlinear}, the unique solution $\Theta \in X_{\Theta}$ of this problem satisfies
\begin{equation} \label{Heat_bound}
\begin{aligned}
&\|\Theta(t)\|_{\Honetwo}^2+\|\Theta_t(t)\|^2_{L^2}+\int_0^t (\|\Theta_{tt}\|^2_{H^{-1}}+\|\Theta_t\|^2_{H^1})\ds\\
\leq&\, C_T (\|\Theta_0\|_{\Honetwo}^2+\|\tilde{f}(0)\|^2_{L^2}+\int_0^t \|\tilde{f}_t\|^2_{L^2}\ds)
\end{aligned}
\end{equation}
for all $t \in [0,T]$; see also~\cite[Ch.\ 2, Theorem 3.2]{temam2012infinite}. Thanks to the assumed properties of the mapping $\mathcal{Q}$, we have
\[
\begin{aligned}
\|\tilde{f}\|_{L^2(L^2)} \lesssim \|f_2\|_{L^2(L^2)}+  \|p_t\|_{L^\infty(L^\infty)}\|p_t\|_{L^2(L^2)}+C(T, \Omega, \Thetaa).
\end{aligned}
\]
Further,
\[
\begin{aligned}
\|\tilde{f}_t\|_{L^2(L^2)} \lesssim&\, \|\partial_t f_{2}\|_{L^2(L^2)}+ \|p_t\|_{L^2(L^\infty)}\| p_{tt}\|_{L^\infty(L^2)}.
\end{aligned}
\]
Thus, by the embedding $H^1(0,T) \hookrightarrow C[0,T]$, from \eqref{Heat_bound} we have
\begin{equation}\label{Theta_Energy_Estimate}
\begin{aligned}
&\|\Theta(t)\|_{\Honetwo}^2+\|\Theta_t(t)\|^2_{L^2}+\int_0^t (\|\Theta_{tt}\|^2_{H^{-1}}+\|\Theta_t\|^2_{H^1})\ds\\
\leq&\, C_T \big(\|\Theta_0\|_{H^2}^2+\|f_2\|^2_{H^1(L^2)}+ \|p_t\|^2_{L^\infty(L^\infty)}\|p_t\|^2_{L^2(L^2)}+ \|p_t\|_{L^2(L^\infty)}^2\| p_{tt}\|^2_{L^\infty(L^2)}+1\big),
\end{aligned}  
\end{equation} 
as claimed. This finishes the proof of Proposition \ref{Prop:LinWellp}.
\end{proof}
\section{Local well-posedness of the nonlinear problem }  \label{Sec:NonLinProblem}
To prove local well-posedness of the coupled Westervelt--Pennes model, we intend to rely on Banach's fixed
point theorem. To this end, let us introduce the fixed-point mapping $\mathcal{T}: (p^*, \Theta^*) \mapsto (p, \Theta)$, which associates 
\[
(p_*, \Theta_*) \in B \subset X_T:=X_p \times X_\Theta,
\]  
where $B$ will be a suitably chosen ball in $X_T$, with the solution $(p, \Theta) \in X_p \times X_{\Theta}$ of
\begin{equation} \label{System_Fixed_Point}
\left\{ \begin{aligned}
&(1-2k(\Theta_*)p_*)p_{tt}-q(\Theta_*)\Delta p - b \Delta p_t = 2k(\Theta_*)p_{*t}^2, \qquad &&\text{in} \ \Omega \times (0,T), \\[1mm] 
& 	\rhoa \Ca\Theta_t -\kappaa\Delta \Theta+ \rhob \Cb W(\Theta-\Thetaa) = \mathcal{Q}(p_t), \qquad &&\text{in} \ \Omega \times (0,T),
\end{aligned} \right.
\end{equation}
with the boundary \eqref{coupled_problem_BC} and initial \eqref{coupled_problem_IC} conditions. Our main results reads as follows.
\begin{theorem}\label{Thm:NlWellP}
Let $T>0$ and 
\begin{equation}
(p_0, p_1) \in \Honethree \times \Honetwo, \quad \Theta_0 \in \Honetwo.  
\end{equation}
There exists $\delta=\delta(T)>0$, such that if
\begin{equation}
\mathcal{E}[p](0) \leq \delta,
\end{equation}
then there exist a unique solution $(p, \Theta)$ of \eqref{coupled_problem} in $X_T$. Furthermore, the solution depends continuously on the data with respect to $\|\cdot\|_{X_T}$.
\end{theorem}  
\begin{proof} 
As already announced, we intend to rely on Banach's fixed-point theorem to arrive at the claim. To facilitate the fixed-point argument, we define the pressure and temperature norms:
\begin{equation}
\begin{aligned}
\Vert p\Vert_{X_p}=&\,\begin{multlined}[t] \Vert  p \Vert_{L^\infty(H^3)}+\Vert  p_t \Vert_{L^\infty(H^2)}+\|\nabla \Delta p_t\|_{L^2(L^2)} +\Vert  p_{tt} \Vert_{L^\infty( L^2)}\\+\Vert  p_{tt} \Vert_{L^2(H^1(\Omega))}
+\Vert  p_{ttt} \Vert_{L^2(H^{-1}(\Omega))}
 \end{multlined}
\end{aligned}
\end{equation}
and 
\begin{equation}
\begin{aligned}
\Vert \Theta\Vert_{X_\Theta}=&\,\begin{multlined}[t] \Vert\Theta \Vert_{L^\infty(H^2)}+\Vert \Theta_t \Vert_{L^\infty(L^2)}
+\Vert \Theta_t \Vert_{L^2(H^1)}+\Vert \Theta_{tt} \Vert_{L^2(H^{-1})}.\end{multlined}
\end{aligned}
\end{equation}
We can then also define the combined norm as follows:
\begin{equation}
\Vert (p,\Theta)\Vert_{X_T}=\Vert p\Vert_{X_p}+\Vert \Theta\Vert_{X_\Theta}. 
\end{equation}  
To have an equivalence between this norm and the energies, we introduce the total pressure energy $\mathbb{E}[p]$ as 
\begin{equation}
\mathbb{E}[p](T)=\sup_{t \in (0,T)}\mathcal{E}[p](t)+\sup_{t \in (0,T)}\|\Delta p_t(t)\|^2_{L^2}
\end{equation}
and the associated dissipation rate as 
\begin{equation}
\mathbb{D}(t)=\mathcal{D}[p](t)+\int_0^t(\|p_{ttt}(s)\|^2_{H^{-1}}+\|\nabla \Delta p_t(s)\|^2_{L^2})\ds.  
\end{equation}
Then on account of Assumption \ref{Assumptions:LinProblem}, there exist positive constants  $C_1, \ldots, C_4$, such that  
\begin{equation} \label{eqv_temp}
\begin{aligned}
C_1\left(\mathbb{E}[p](T)+\mathbb{D}[p](T)\right)
\leq\, \Vert p \Vert_{X_p}^2
\leq\, C_2 \left(\mathbb{E}[p](T)+\mathbb{D}[p](T)\right)
\end{aligned}
\end{equation}
and
\begin{equation} \label{eqv_temp}
\begin{aligned}
C_3\left(\sup_{t \in (0,T)}\mathcal{E}[\Theta](t)+\mathcal{D}[\Theta](T)\right)
\leq\, \Vert \Theta \Vert_{X_\Theta}^2
\leq\, C_4 \left(\sup_{t \in (0,T)}\mathcal{E}[\Theta](t)+\mathcal{D}[\Theta](T)\right). 
\end{aligned}
\end{equation}
We next introduce a ball in $X_T$:
\begin{equation}
\begin{aligned}
B= \left\{(p_*, \Theta_*)\in X_T:\right.&\, \|p_*\|_{L^\infty(L^\infty)} \leq \gamma < \frac{1}{2k_1}, \quad
 \|p_*\|_{X_p} \leq R_1,\\ &\left. \|\Theta_*\|_{X_\Theta} \leq R_2, \quad
(p_*, p_{*t}, \Theta_*)_{\vert t=0}=(p_0, p_1, \Theta_0) \right\},
\end{aligned}
\end{equation}
where the radii $R_1>0$ and $R_2>0$ are to be determined by the proof. The constant $k_1>0$ is such that
\[
|k(\Theta)| \leq k_1;
\] 
cf.\ assumption \eqref{Assumption_k_1}. In the course of the proof we will impose a smallness condition on the pressure, but not on the temperature data, which is why we have introduced two different radii here. \\ 
\indent Note that the solution of the linear problem with $\alpha=r=1$ and $f_1=f_2=0$, belongs to this ball if $\delta>0$ is small enough and $R_2$ large enough, so that\\
\[R_1^2 \geq C_T \delta \geq C_T\mathcal{E}[p](0), \qquad R_2^2 \geq \tilde{C}_T(\|\Theta_0\|^2_{\Honetwo}+\delta^2+1),\]
so this set is non-empty. We consider the ball to be equipped with the distance 
\begin{equation}    
d[(p_1,p_2), (\Theta_1, \Theta_2)]=\Vert p_1-p_2\Vert_{X_p}+\Vert \Theta_1-\Theta_2\Vert_{X_\Theta}.
\end{equation}
 Then $(B, d)$ is a complete metric space. We first prove that $\mathcal{T}$ is a self-mapping.
\begin{lemma} \label{Lemma2}
For sufficiently small $R_1$ and $\delta$, it holds that $\mathcal{T}(B) \subset B$.
	\end{lemma}
\begin{proof}
 We wish to rely on the well-posedness result from the previous section. To this end, we set
\begin{equation}\label{Coeff_Nonli_Problem}
\alpha(x,t)= 1-2k(\Theta_*)p_*, \quad r(x,t)=q(\Theta_*), \quad f_1(x,t)= 2k(\Theta_*)p_{*t}^2, \quad f_2(x,t)=0.
\end{equation}  
to fit problem \eqref{System_Fixed_Point} into the framework of Proposition \ref{Prop:LinWellp}. We next verify Assumption~\ref{Assumptions:LinProblem} on these functions. Since
\begin{equation}
\begin{aligned}
\|2k(\Theta_*)p_*\|_{L^\infty(L^\infty)} \leq 2k_1 \|p_*\|_{L^\infty(L^\infty)} \leq 2k_1 \gamma
\end{aligned}  
\end{equation}
we have  
\[
0<\alpha_0=1-2k_1 \gamma \leq \alpha(x,t)=1-2k(\Theta_*)p_* \leq 1+2k_1 \gamma=\alpha_1
\]
and so the non-degeneracy condition is fulfilled. Further, by the embeddings $H^1(\Omega) \hookrightarrow L^3(\Omega)$ and $H^2(\Om) \hookrightarrow L^\infty(\Omega)$, we have
\begin{equation}
\begin{aligned}
\|\alpha\|_{L^\infty(W^{1,3})}\lesssim&\,  \|1-2k(\Theta_*)p_*\|_{L^\infty(L^3)}+ \|\nabla (k(\Theta_*))p_*\|_{L^\infty(L^3)}+ \|k(\Theta_*)\nabla p_*\|_{L^\infty(L^3)}\\
\lesssim&\,\begin{multlined}[t] 1+k_1\|p_*\|_{L^\infty(H^1)}+\| k'(\Theta_*)\|_{L^\infty(L^\infty)} \|\nabla \Theta_*\|_{L^\infty(L^3)}\|p_*\|_{L^\infty(H^2)}\\+k_1\|p_*\|_{L^\infty(H^1)}.\end{multlined}
\end{aligned}  
\end{equation}
From here and properties \eqref{Assumption_k_2} of the function $k$, it follows that
\begin{equation}
\begin{aligned}
\|\alpha\|_{L^\infty(W^{1,3})}\lesssim&\,  1+R_1+(1+R_2^{\gamma_2+1})R_1R_2.
\end{aligned}  
\end{equation}
Again by the embedding $H^1(\Omega) \hookrightarrow L^3(\Omega)$ and properties of the function $k$, it holds that
\begin{equation}
\begin{aligned}
\|\alpha_t\|_{L^{2}(L^3)} =&\, \|-2k(\Theta_*)p_{*t}-2k'(\Theta_*)\Theta_{*t}p_*\|_{L^{2}(L^3)} \\
\lesssim&\,q_0^{-1}\|\nabla p_{*t}\|_{L^2(L^2)}+q_0^{-2}(1+\|\Theta_*\|^{\gamma_2+1}_{L^\infty(L^\infty)})\|\nabla \Theta_{*t}\|_{L^2(L^2)}\|p_*\|_{L^{\infty}(L^\infty)},
\end{aligned}
\end{equation}
which implies
\begin{equation}
\begin{aligned}
\|\alpha_t\|_{L^{2}(L^3)} 
\lesssim\, R_1+(1+R_2^{\gamma_2+1})R_1R_2.
\end{aligned}
\end{equation}
Similarly,
\[
\begin{aligned}
\|\alpha_t\|_{L^{\infty}(L^2)} =&\, \|-2k(\Theta_*)p_{*t}-2k'(\Theta_*)\Theta_{*t}p_*\|_{L^{\infty}(L^2)} \\
\lesssim&\,q_0^{-1}\|\nabla p_{*t}\|_{L^\infty(L^2)}+q_0^{-2}(1+\|\Theta_*\|^{\gamma_2+1}_{L^\infty(L^\infty)})\|\Theta_{*t}\|_{L^\infty(L^2)}\|p_*\|_{L^{\infty}(L^\infty)} \\
\lesssim&\, R_1+(1+R_2^{\gamma_2+1})R_1R_2.
\end{aligned}
\]
We can analogously estimate the function $r$:
\begin{equation}
\begin{aligned}
\|r_t\|_{L^\infty(L^2)} \lesssim&\, \|q'(\Theta_*)\|_{L^\infty(L^\infty)} \|\Theta_{t*}\|_{L^\infty(L^2)},\\
\|\nabla r\|_{L^\infty(L^4)}=&\,\|q'(\Theta_*) \nabla \Theta_*\|_{L^\infty(L^4)} \lesssim \|q'(\Theta_*)\|_{L^\infty(L^\infty)} \|\Theta_*\|_{L^\infty(H^2)},
\end{aligned}
\end{equation} 
and thus
\begin{equation}
\begin{aligned}
\|r_t\|_{L^\infty(L^2)} \lesssim 1+(1+R_2^{\gamma_1+1})R_2,\qquad
\|r\|_{L^{\infty}(W^{1,4})}\lesssim\,1+(1+R_2^{\gamma_1+1})R_2.
\end{aligned}
\end{equation} 
We can further estimate the source term in the pressure equation as follows:
\begin{equation}
\begin{aligned}
&\|f_1\|_{L^2(H^1)}+\|\partial_t f_1\|_{L^2(H^{-1})}\\
 \lesssim&\, \|k(\Theta_*)p_{*t}^2\|_{L^2(H^1)}+\|\partial_t(k(\Theta_*)p_{*t}^2)\|_{L^2(H^{-1})} \\
\lesssim&\,\begin{multlined}[t] \|k'(\Theta_*)\nabla \Theta_* p_{*t}^2\|_{L^2(L^2)}+\| k(\Theta_*) p_{*t} \nabla p_{*t}\|_{L^2(L^2)}\\+\|k(\Theta_*)p_{*t}^2\|_{L^2(L^2)}+\|k'(\Theta_*)\Theta_{t*} p_{*t}^2\|_{L^2(H^{-1})}+\|k(\Theta_*)p_{*t}p_{*tt}\|_{L^2(H^{-1})}.
\end{multlined}
\end{aligned}  
\end{equation}
By using the embedding $L^{6/5}(\Omega) \hookrightarrow H^{-1}(\Omega)$ and the inequality
\begin{equation} \label{ineq:uvw}
\|uvw\|_{L^{6/5}} \leq \|u\|_{L^2}\|v\|_{L^3}\|w\|_{L^\infty}
\end{equation}
we then further have
\begin{equation} \label{est_f}
\begin{aligned}
&\|f_1\|_{L^2(H^1)}+\|\partial_t f_1\|_{L^2(H^{-1})}\\
\lesssim&\,\begin{multlined}[t] \|k'(\Theta_*)\|_{L^\infty(L^\infty)}\|\nabla \Theta_*\|_{L^\infty(L^6)}\| p_{*t}^2\|_{L^2(L^3)}+k_1\|p_{*t}\|_{L^\infty(L^4)} \|\nabla p_{*t}\|_{L^2(L^4)}\\+k_1\|p_{*t}^2\|_{L^2(L^2)}+\|k'(\Theta_*)\|_{L^\infty(L^\infty)}\|\Theta_{t*}\|_{L^2(L^3)}\|p_{*t}^2\|_{L^\infty(L^2)}\\+k_1\|p_{*t}\|_{L^\infty(L^3)}\|p_{*tt}\|_{L^2(L^2)}.
\end{multlined}
\end{aligned}  
\end{equation}
Thus,
\begin{equation}\label{f_H_1_H-1}
\begin{aligned}
\|f_1\|_{L^2(H^1)}+\|\partial_t f_1\|_{L^2(H^{-1})}
\lesssim\,(1+R_2^{\gamma_2+1})R_2R_1^2+R_1^2.
\end{aligned}  
\end{equation}
On account of Proposition~\ref{Prop:LinWellp}, the mapping $\mathcal{T}$ is well-defined, and, furthermore,
\begin{equation} \label{energy_est_selfmapping}
\begin{aligned}
&\mathcal{E}[p](t)+ \|\Delta p_t(t)\|^2_{L^2}  +\int_0^t\mathcal{D}[p](s)\ds+\int_0^t\|p_{ttt}(s)\|^2_{H^{-1}}\ds\\
\lesssim &\, \mathcal{E}[p](0)\exp{\left(\int_0^t (1+\Lambda(s))\ds\right)}+\int_0^t\exp{\left(\int_s^t (1+\Lambda(\sigma))\textup{d}\sigma\right)}\mathbb{F}(s)\ds
\end{aligned}
\end{equation}
a.e.\ in time, with $\Lambda(t)$ and $\mathbb{F}(t)$ defined in \eqref{def_Lamba} and \eqref{F_Terms}, respectively; that is,  
\begin{equation} 
\Lambda(t)=\left\Vert r_t(t)\right\Vert_{L^2}^2+\|\nabla r(t)\|_{L^4}+\left\Vert \alpha_t(t)\right\Vert_{L^2}+\Vert \alpha_t(t)\Vert_{L^3}^2+\left\|\nabla \alpha(t)\right\|_{L^3}^2 
\end{equation}  
and    
\begin{equation}
\mathbb{F}(t)=\Vert f_1(t)\Vert_{H^1}^2+(1+\left\|\nabla \alpha(t)\right\|_{L^3}^2)\Vert \partial_tf_1(t)\Vert_{H^{-1}}^2.
\end{equation}
By our calculations above, we immediately have
\begin{equation}
\begin{aligned}
\|\Lambda\|_{L^1(0,t)} \leq C_1(R_1, R_2, T),
\end{aligned}
\end{equation}
where $C_1=C_1(T, R_1, R_2)$ is a positive constant that depends on $T, R_1$, and $R_2$. Furthermore, by relying on \eqref{est_f}, we obtain 
\begin{equation}
\begin{aligned}
\|\mathbb{F}\|_{L^1(0,t)} \lesssim&\, (1+\left\|\nabla \alpha\right\|_{L^\infty(L^3)}^2)(\Vert f_1\Vert_{L^2(H^1)}^2+\Vert \partial_tf_1\Vert_{L^2(H^{-1})}^2) \\
\lesssim& (1+R_1^2+(1+R_2^{2\gamma_2+2})R_1^2R_2^2)\left\{(1+R_2^{2\gamma_2+2})R_2^2R_1^4+R_1^4\right\}.
\end{aligned}
\end{equation}
Altogether, from \eqref{energy_est_selfmapping} and the above bounds, we have
\begin{equation} \label{R_bound}
\begin{aligned}
\Vert p \Vert_{X_p}^2\lesssim \begin{multlined}[t] \delta\exp(C_1(R_1, R_2, T)T)+\exp(C_1(R_1, R_2, T)T)\,R_1^4\,C_2(R_1, R_2). \end{multlined}
\end{aligned}
\end{equation}
Thus, from \eqref{R_bound}, by decreasing $R_1$ and $\delta$, we can achieve that
\begin{equation}
\begin{aligned}
\|p\|_{X_p}^2 \leq  R_1^2.
\end{aligned}
\end{equation}
Further, by the embedding $H^2(\Omega) \hookrightarrow L^\infty(\Omega)$, we know that
\begin{equation}
\begin{aligned}
\|p\|^2_{L^\infty(L^\infty)} \lesssim \|\Delta p\|_{L^\infty(L^2)}^2  \lesssim  \|p\|_{X_p}^2,
\end{aligned}
\end{equation}
which we can then bound by $\gamma \in (0, 1/(2k))$ by possibly additionally reducing $\delta$ and $R_1$. It remains to show that $\|\Theta\|_{X_\Theta} \leq R_2$. Proposition~\ref{Prop:LinWellp} with $f_2=0$ implies that 
\begin{equation}
\begin{aligned}        
&\mathcal{E}[\Theta](t)+\int_0^t\mathcal{D}[\Theta](s)\ds\\
\leq&\, C_T \big(\|\Theta_0\|_{\Honetwo}^2+ \|p_t\|^2_{L^\infty(L^\infty)}\|p_t\|^2_{L^2(L^2)}+ \|p_t\|_{L^2(L^\infty)}^2\| p_{tt}\|^2_{L^\infty(L^2)}+1\big).
\end{aligned}
\end{equation}
With the equivalence of the temperature norm and energy \eqref{eqv_temp}, we have
\begin{equation}
\begin{aligned}
\|\Theta\|^2_{X_\Theta} \leq&\, \begin{multlined}[t] C_T \big(\|\Theta_0\|_{\Honetwo}^2+ \|p_t\|^2_{L^\infty(L^\infty)}\|p_t\|^2_{L^2(L^2)}+ \|p_t\|_{L^2(L^\infty)}^2\| p_{tt}\|^2_{L^\infty(L^2)}+1\big)\end{multlined} \\
\leq&\,  \tilde{C}_T \left(\|\Theta_0\|_{\Honetwo}^2+ 2R_1^4+1\right).
\end{aligned}
\end{equation}
Thus, if we additionally choose $R_2$ large enough, so that 
\[R_2^2 \geq \tilde{C}_T  \left(\|\Theta_0\|_{\Honetwo}^2+ 2R_1^4+1\right),\]
we have $(p, \Theta) \in B$.
\end{proof}
\begin{lemma} \label{Lemma3}
For sufficiently small $R_1$ and $\delta$, the mapping $\mathcal{T}$ is strictly contractive in the topology induced by $\|\cdot\|_{X_T}$.
\end{lemma}	
\begin{proof}
 To prove contractivity, take any $(p^{(1)}_*, \Theta^{(1)}_*)$ and $(p^{(2)}_*, \Theta^{(2)}_*)$ from $B$. Denote their images by $(p^{(1)}, \Theta^{(1)})=\mathcal{T}(p^{(1)}_*, \Theta^{(1)}_*)$ and $(p^{(2)}, \Theta^{(2)})=\mathcal{T}(p^{(2)}_*, \Theta^{(2)}_*)$. We introduce the differences
\begin{equation}
\begin{aligned}
\overline{p}=&\, p^{(1)}-p^{(2)}, \qquad  \overline{p}_*=p^{(1)}_*-p^{(2)}_*, \\
\overline{\Theta}=&\, \Theta^{(1)}-\Theta^{(2)}, \qquad  \overline{\Theta}^*=\Theta^{(1)}_*-\Theta^{(2)}_*.
\end{aligned}  
\end{equation}
Our goal now is to prove that 
\begin{equation}\label{Contra_map}
\begin{aligned}
&\Vert \mathcal{T}(p^{(1)}_*, \Theta^{(1)}_*)-\mathcal{T}(p^{(2)}_*, \Theta^{(2)}_*)\Vert_{X_T}\\
\leq&\, R_1C(T,R_1,R_2) \Vert (p^{(1)}_* -p^{(2)}_*, \Theta^{(1)}_*-\Theta^{(2)}_*)\Vert_{X_T}, 
\end{aligned}
\end{equation}
where $C$ is a positive constant that depends on $T, R_1$, and $R_2$. Observe that  
 $(\overline{p}, \overline{\Theta})$ solves the following problem:
\begin{equation} \label{System_Contract}
\left\{ \begin{aligned}
& \begin{multlined}(1-2k(\Theta^{(1)}_*)p^{(1)}_*)\overline{p}_{tt}-q(\Theta^{(1)}_*)\Delta \overline{p} - b \Delta \overline{p}_t= \overline{f}_1 \end{multlined}\quad &&\text{in} \ \Omega \times (0,T), \\[1mm]
& 	\rhoa \Ca\overline{\Theta}_t -\kappaa\Delta \overline{\Theta}+ \rhob \Cb W \overline{\Theta}= \overline{f}_2 \quad &&\text{in} \ \Omega \times (0,T),\\
& \overline{p}= \overline{\Theta}=0,\quad &&\text{on} \ \partial \Omega \times (0,T),\\
& \overline{p}(x, 0)=\overline{p}_t(x, 0)=\overline{\Theta}(x, 0)=0,\quad &&\text{in} \ \Omega ,
\end{aligned} \right.
\end{equation}
with the right-hand sides
\begin{equation} \label{overline_f1}
\begin{aligned}
\overline{f}_1=&\, \begin{multlined}[t] \left\{2k(\Theta^{(1)}_*)p^{(1)}_*-2k(\Theta^{(2)}_*)p^{(2)}_*\right\}p_{*tt}^{(2)}+\left\{q(\Theta^{(1)}_*)-q(\Theta^{(2)}_*)\right\}\Delta p^{(2)}_*\\+2k(\Theta^{(1)}_*)(p_{*t}^{(1)})^2-2k(\Theta^{(2)}_*)(p_{*t}^{(2)})^2
\end{multlined}
\end{aligned}
\end{equation}
and
\begin{equation} \label{overline_f2}
\overline{f}_2= \mathcal{Q}(p_{*t}^{(1)})-\mathcal{Q}(p_{*t}^{(2)}).
\end{equation}  
We can rearrange the acoustic source term $\overline{f}_1$ as follows:
\begin{equation}\label{f_1_Formula}
\begin{aligned}
&\overline{f}_1\\
=&\, \begin{multlined}[t] 2\Big\{k(\Theta^{(1)}_*)-k(\Theta^{(2)}_*)\Big\}p^{(1)}_*p_{*tt}^{(2)}+2k(\Theta^{(2)}_*)\overline{p}_*p_{*tt}^{(2)}+\Big\{q(\Theta^{(1)}_*)-q(\Theta^{(2)}_*)\Big\}\Delta p^{(2)}_*\\+2\left\{k(\Theta^{(1)}_*)-k(\Theta^{(2)}_*)\right\}(p_{*t}^{(1)})^2+2k(\Theta^{(2)}_*)\overline{p}_{*t}(p_{*t}^{(1)}+p_{*t}^{(2)})
\end{multlined} \\
=&\, \begin{multlined}[t] 2\left\{k(\Theta^{(1)}_*)-k(\Theta^{(2)}_*)\right\}\left(p^{(1)}_*p_{*tt}^{(2)}+(p_{*t}^{(1)})^2\right)+\left\{q(\Theta^{(1)}_*)-q(\Theta^{(2)}_*)\right\}\Delta p^{(2)}_*\\+2k(\Theta^{(2)}_*)\left(\overline{p}_*p_{*tt}^{(2)}+\overline{p}_{*t}(p_{*t}^{(1)}+p_{*t}^{(2)})\right)
\end{multlined} \\  
:=&\,\overline{f}_{11}+\overline{f}_{12}+\overline{f}_{13}  
\end{aligned}
\end{equation}    
and next wish to show that it satisfies Assumption~\ref{Assumptions:LinProblem}.

\subsubsection*{The estimate of $\Vert \overline{f}_1\Vert_{L^2(H^1)}$} 
 Note that since $\overline{f}_1=0$ on $\partial{\Omega}$, it is sufficient to estimate $\Vert \nabla \overline{f}_1\Vert_{L^2(L^2)}$. We first estimate the $\overline{f}_{11}$ contribution, that is
 \[
\overline{f}_{11}= 2\left\{k(\Theta^{(1)}_*)-k(\Theta^{(2)}_*)\right\}\left(p^{(1)}_*p_{*tt}^{(2)}+(p_{*t}^{(1)})^2\right).
 \] By H\"older's inequality, we have 
 \begin{equation}\label{nabla_f_11_Estimate}
\begin{aligned}  
\Vert \nabla \overline{f}_{11}\Vert_{L^2(L^2)}\lesssim &\Vert \, k(\Theta^{(1)}_*)-k(\Theta^{(2)}_*)\Vert_{L^\infty(L^\infty)}\|\nabla(p^{(1)}_*p_{*tt}^{(2)}+(p_{*t}^{(1)})^2)\|_{L^2(L^2)}\\
&+ \|\nabla(k(\Theta^{(1)}_*)-k(\Theta^{(2)}_*))\|_{L^\infty(L^4)}\|p^{(1)}_*p_{*tt}^{(2)}+(p_{*t}^{(1)})^2\|_{L^2(L^4)}.
\end{aligned}
\end{equation}
Recalling properties \eqref{Assumption_k_1}  and \eqref{Assumption_k_2} of the function $k$, and using the algebraic inequality:
\begin{equation}
(A+B)^\nu\leq \max\{1, 2^\nu\}(A^\nu+B^\nu),\quad \text{for}\quad A,\, B\geq 0,\, \nu>0, 
\end{equation}
 we have
\begin{equation}\label{First_Term_f_1}
\begin{aligned}
&\Vert \, k(\Theta^{(1)}_*)-k(\Theta^{(2)}_*)\Vert_{L^\infty(L^\infty)}\\
= &\,\Big\Vert (\Theta^{(1)}_*-\Theta^{(2)}_*)\int_0^1 k^\prime (\Theta^{(1)}_*+\tau(\Theta^{(1)}_*-\Theta^{(2)}_*))\, \textup{d}\tau \Big\Vert_{L^\infty (L^\infty)}\\  
\lesssim  &\, \Vert\Theta^{(1)}_*-\Theta^{(2)}_*\Vert_{L^\infty(L^\infty)}\Big(1+\Vert \Theta^{(1)}_*+\tau(\Theta^{(1)}_*-\Theta^{(2)}_*) \Vert_{L^\infty(L^\infty)}^{\gamma_2+1}\Big)\\
\lesssim  &\,  \Vert (p^{(1)}_* -p^{(2)}_*, \Theta^{(1)}_*-\Theta^{(2)}_*)\Vert_{X_T}\left\{1+\Vert  \Theta^{(1)}_*\Vert_{X_\Theta}^{\gamma_2+1}+ \Vert  \Theta^{(2)}_*\Vert_{X_\Theta}^{\gamma_2+1}\right\}.
\end{aligned} 
\end{equation}   
We also have, by using the embeddings $H^1(\Om)\hookrightarrow L^4(\Om)$ and $H^2(\Om)\hookrightarrow L^\infty(\Om)$, the following estimate:
\begin{equation}\label{p_p_t_Term_Estimate}
\begin{aligned}
&\|\nabla(p^{(1)}_*p_{*tt}^{(2)}+(p_{*t}^{(1)})^2)\|_{L^2(L^2)}\\
\lesssim&\,\begin{multlined}[t]\Vert \nabla p^{(1)}_*\Vert_{L^\infty(L^4)}\Vert p_{*tt}^{(2)}\Vert_{L^2(L^4)}+\Vert  p^{(1)}_*\Vert_{L^\infty(L^\infty)}\Vert \nabla p_{*tt}^{(2)}\Vert_{L^2(L^2)} \\
+\Vert p^{(1)}_{*t}\Vert_{L^2(L^\infty)}\Vert \nabla p^{(1)}_{*t}\Vert_{L^\infty (L^2)} \end{multlined} \\   
\lesssim &\, \begin{multlined}[t]\Vert \Delta p^{(1)}_*\Vert_{L^\infty(L^2)}\Vert \nabla p_{*tt}^{(2)}\Vert_{L^2(L^2)}+\Vert \Delta  p^{(1)}_*\Vert_{L^\infty(L^2)}\Vert \nabla p_{*tt}^{(2)}\Vert_{L^2(L^2)}\\
+\Vert\Delta p_{*t}^{(1)}\Vert_{L^2 (L^2)}\Vert \nabla p_{*t}^{(1)}\Vert_{L^\infty (L^2)}. \end{multlined}
\end{aligned}   
\end{equation}
Thus, from \eqref{p_p_t_Term_Estimate} it follows that
 \begin{equation}\label{p_p_t_Term_Estimate_2}
\begin{aligned}
\|\nabla(p^{(1)}_*p_{*tt}^{(2)}+(p_{*t}^{(1)})^2)\|_{L^2(L^2)}
\lesssim \, \Vert p^{(1)}_*\Vert_{X_p}^2. 
\end{aligned}
\end{equation}
Further, we know that 
\begin{equation}
\begin{aligned}
\nabla(k(\Theta^{(1)}_*)-k(\Theta^{(2)}_*))=&\,k^\prime (\Theta^{(1)}_*)\nabla \Theta^{(1)}_*-k^\prime (\Theta^{(2)}_*)\nabla \Theta^{(2)}_*\\
=&\, k^\prime (\Theta^{(1)}_*) \nabla (\Theta^{(1)}_*-\Theta^{(2)}_*)+\nabla \Theta^{(2)}_*(k^\prime (\Theta^{(1)}_*)-k^\prime (\Theta^{(2)}_*))
\end{aligned}
\end{equation}
and
\begin{equation}
\begin{aligned}
k^\prime (\Theta^{(1)}_*)-k^\prime (\Theta^{(2)}_*)= (\Theta^{(1)}_*-\Theta^{(2)}_*)\int_0^1 k'' (\Theta^{(1)}_*+\tau(\Theta^{(1)}_*-\Theta^{(2)}_*))\, \textup{d}\tau.
\end{aligned}
\end{equation}
By keeping in mind properties \eqref{Assumption_k_1} and  \eqref{Assumption_k_2} of the function $k$, this implies that 
   \begin{equation}
\begin{aligned}
&\|\nabla(k(\Theta^{(1)}_*)-k(\Theta^{(2)}_*))\|_{L^\infty(L^4)}\\
\lesssim&\, \begin{multlined}[t]\left(1+ \|\Theta^{(1)}_*\|^{\gamma_2+1}_{L^\infty(L^\infty)}\right) \Vert\nabla (\Theta^{(1)}_* -\Theta^{(2)}_*)\Vert_{L^\infty(L^4)}\\
+\left(1+\Vert \Theta^{(1)}_*\Vert_{L^\infty(L^\infty)}^{\gamma_2}+\Vert \Theta^{(2)}_*\Vert_{L^\infty(L^\infty)}^{\gamma_2} \right) 
\Vert \nabla \Theta^{(2)}_*\Vert_{L^\infty(L^4)}\|\Theta^{(1)}_*-\Theta^{(2)}_*\|_{L^\infty(L^\infty)} \end{multlined}
\end{aligned}
\end{equation} 
and thus
  \begin{equation}\label{nabla_k_Estimate}
\begin{aligned}
&\|\nabla(k(\Theta^{(1)}_*)-k(\Theta^{(2)}_*))\|_{L^\infty(L^4)}\\
\lesssim&\, \begin{multlined}[t] \left\{1+\Vert  \Theta^{(1)}_*\Vert_{X_\Theta}^{\gamma_2}+ \Vert  \Theta^{(2)}_*\Vert_{X_\Theta}^{\gamma_2}+\Vert  \Theta^{(1)}_*\Vert_{X_\Theta}^{\gamma_2+1}\right\} 
\Vert (p^{(1)}_* -p^{(2)}_*, \Theta^{(1)}_*-\Theta^{(2)}_*)\Vert_{X_T}. \end{multlined}
\end{aligned}
\end{equation} 
To obtain a bound on $\nabla \overline{f}_{11}$, we note that
\begin{equation}\label{Fourth_Term_Estimate}
\begin{aligned}
&\|p^{(1)}_*p_{*tt}^{(2)}+(p_{*t}^{(1)})^2\|_{L^2(L^4)}\\
\lesssim&\, \Vert p^{(1)}_*\Vert_{L^2(L^\infty)} \Vert p_{*tt}^{(2)} \Vert_{L^2(L^4)}+\Vert p_{*t}^{(2)} \Vert_{L^\infty(L^4)}\Vert p_{*t}^{(2)} \Vert_{L^2(L^\infty)}\\
\lesssim&\, \sqrt{T} \Vert \Delta p^{(1)}_*\Vert_{L^\infty(L^2)}\Vert \nabla p_{*tt}^{(2)} \Vert_{L^2(L^2)}+\Vert \nabla p_{*t}^{(1)} \Vert_{L^\infty(L^2)}\Vert \Delta p_{*t}^{(1)} \Vert_{L^2(L^2)}\\
\lesssim& \,(1+\sqrt{T}) \left(\Vert p^{(1)}_*\Vert_{X_p}^2+\Vert p^{(2)}_*\Vert_{X_p}^2\right).
\end{aligned}  
 \end{equation}
Plugging the derived estimates into \eqref{nabla_f_11_Estimate} yields  
\begin{equation}\label{nabla_f_11_Estimate_2}
\begin{aligned}\label{nabla_f_1_1_Main}
\Vert \nabla \overline{f}_{11}\Vert_{L^2(L^2)}\lesssim &\, \begin{multlined}[t](1+\sqrt{T})R^2_1\left(1+R_2^{\gamma_2}+R_2^{\gamma_2+1} \right)
 \Vert (p^{(1)}_* -p^{(2)}_*, \Theta^{(1)}_*-\Theta^{(2)}_*)\Vert_{X_T} \end{multlined}.
\end{aligned}
\end{equation}
We can similarly estimate $\overline{f}_{12}=\left\{q(\Theta^{(1)}_*)-q(\Theta^{(2)}_*)\right\}\Delta p^{(2)}_*$ as follows:
\begin{equation}\label{nabla_f_12_1}
\begin{aligned}
\Vert \nabla \overline{f}_{12}\Vert_{L^2(L^2)}\lesssim&\, \begin{multlined}[t]\left\Vert \nabla( q(\Theta^{(1)}_*)-q(\Theta^{(2)}_*))\right\Vert_{L^\infty(L^4)}\Vert\Delta p^{(2)}_*\Vert_{L^2(L^4)}\\
+ \Vert  q(\Theta^{(1)}_*)-q(\Theta^{(2)}_*)\Vert_{L^\infty(L^\infty)}\Vert\nabla \Delta p^{(2)}_*\Vert_{L^2(L^2)}.\end{multlined}
\end{aligned}
\end{equation}
The first term on the right-hand side of \eqref{nabla_f_12_1} can be estimated analogously to \eqref{nabla_k_Estimate}. Thus we have by recalling Assumption \ref{Assumption_1},   
 \begin{equation}
 \begin{aligned}
&\left\Vert \nabla (q(\Theta^{(1)}_*)-q(\Theta^{(2)}_*))\right\Vert_{L^\infty(L^4)}\\
\lesssim &\, \left\{1+\Vert  \Theta^{(1)}_*\Vert_{X_\Theta}^{\gamma_1+1}+ \Vert  \Theta^{(2)}_*\Vert_{X_\Theta}^{\gamma_1+1}\right\} \Vert (p^{(1)}_* -p^{(2)}_*, \Theta^{(1)}_*-\Theta^{(2)}_*)\Vert_{X_T}.
\end{aligned}
\end{equation}
By using the embedding $H^1(\Om)\hookrightarrow L^4(\Om)$, we obtain 
\begin{equation}
\Vert\Delta p^{(2)}_*\Vert_{L^2(L^4)}\lesssim \Vert\Delta p^{(2)}_*\Vert_{L^2(L^2)}+\Vert\Delta\nabla  p^{(2)}_*\Vert_{L^2(L^2)}\lesssim \sqrt{T} \Vert p^{(2)}_*\Vert_{X_p}.
\end{equation}
We also have as in \eqref{First_Term_f_1}, 
\begin{equation}\label{q_L_infty}
\begin{aligned}
 &\left\Vert  q(\Theta^{(1)}_*)-q(\Theta^{(2)}_*)\right\Vert_{L^\infty(L^\infty)}\\
 \lesssim&\, \left\{1+\Vert \Theta^{(1)}_*\Vert_{X_\Theta}^{\gamma_1+1}+ \Vert  \Theta^{(2)}_*\Vert_{X_\Theta}^{\gamma_1+1}\right\}\Vert (p^{(1)}_* -p^{(2)}_*, \Theta^{(1)}_*-\Theta^{(2)}_*)\Vert_{X_T}.
 \end{aligned}
\end{equation}
Consequently, we obtain from above the following estimate: 
\begin{equation}\label{nabla_f_12_2}
\begin{aligned}
\Vert \nabla \overline{f}_{12}\Vert_{L^2(L^2)}\lesssim &\,\begin{multlined}[t](1+\sqrt{T})R_1\left(1+R_2^{\gamma_1}+R_2^{\gamma_1+1}+R_2^{2\gamma_1+2} \right)\\
\times\Vert (p^{(1)}_* -p^{(2)}_*, \Theta^{(1)}_*-\Theta^{(2)}_*)\Vert_{X_T}. \end{multlined}
\end{aligned}
\end{equation} 
Next we estimate $\overline{f}_{13}=2k(\Theta^{(2)}_*)\left(\overline{p}_*p_{*tt}^{(2)}+\overline{p}_{*t}(p_{*t}^{(1)}+p_{*t}^{(2)})\right)$.  We note that
\begin{equation}\label{nabla_f_13}
\begin{aligned}
\Vert\nabla \overline{f}_{13} \Vert_{L^2(L^2)}
\lesssim&\,\begin{multlined}[t] \Vert k'(\Theta^{(2)}_*)\nabla \Theta^{(2)}_* \Vert_{L^\infty(L^4)} \\
\times\Big( \Vert \overline{p}\Vert_{L^\infty (L^\infty )}\Vert p_{*tt}^{(2)} \Vert_{L^2(L^4)} +\Vert \overline{p}_t\Vert_{L^\infty (L^4 )}(\Vert p_{*t}^{(1)} \Vert_{L^2(L^\infty)}+\Vert p_{*t}^{(2)} \Vert_{L^2(L^\infty)})\Big) \\
+ \Vert k(\Theta^{(2)}_*) \Vert_{L^\infty(L^\infty)}\Vert  \nabla (\overline{p}_*p_{*tt}^{(2)}+\overline{p}_{*t}(p_{*t}^{(1)}+p_{*t}^{(2)}))\Vert_{L^2(L^2)}. \end{multlined}
\end{aligned}
\end{equation}
Using properties \eqref{Assumption_k_2} of the function $k$, we can bound the first term on the right:
\begin{equation}\label{nabla_f13_First_Term}
\begin{aligned}
\Vert k'(\Theta^{(2)}_*)\nabla \Theta^{(2)}_* \Vert_{L^\infty(L^4)} \lesssim&\, \Vert k'(\Theta^{(2)}_*)\Vert_{L^\infty(L^\infty)} \Vert \nabla \Theta^{(2)}_* \Vert_{L^\infty(L^4)} \\
\lesssim&\, (1+\Vert \Theta^{(2)}_*\Vert_{L^\infty(L^\infty)}^{\gamma_2+1})\Vert \Theta^{(2)}_*\Vert_{L^\infty(\Honetwo)}\\
\lesssim&\, (1+R_2^{\gamma_2+1})R_2.   
\end{aligned}  
\end{equation}
Further, we have
\begin{equation}\label{nabla_f13_Second_Term}
\begin{aligned}
 &\Vert \overline{p}\Vert_{L^\infty (L^\infty )}\Vert p_{*tt}^{(2)} \Vert_{L^2(L^4)} +\Vert \overline{p}_t\Vert_{L^\infty (L^4 )}(\Vert p_{*t}^{(1)} \Vert_{L^2(L^\infty)}+\Vert p_{*t}^{(2)} \Vert_{L^2(L^\infty)})\\
 \lesssim&\, \begin{multlined}[t]\Vert \Delta \overline{p}\Vert_{L^\infty (L^2 )}\Vert \nabla  p_{*tt}^{(2)} \Vert_{L^\infty (L^2)}
  +\Vert \nabla  \overline{p}_t\Vert_{L^\infty (L^2 )}(\Vert\Delta p_{*t}^{(1)} \Vert_{L^2(L^2)}+\Vert \Delta p_{*t}^{(2)} \Vert_{L^2(L^2)}) \end{multlined}\\
\lesssim &\, R_1 \Vert (p^{(1)}_* -p^{(2)}_*, \Theta^{(1)}_*-\Theta^{(2)}_*)\Vert_{X_T}.
\end{aligned}
\end{equation}
By using the fact that $|k(s)|\lesssim \frac{1}{q_0}$, we find
\begin{equation}\label{Last_Term_nabla_f13}
\begin{aligned}
&\Vert k(\Theta^{(2)}_*) \Vert_{L^\infty(L^\infty)}\Vert  \nabla (\overline{p}_*p_{*tt}^{(2)}+\overline{p}_{*t}(p_{*t}^{(1)}+p_{*t}^{(2)}))\Vert_{L^2(L^2)}\\
\lesssim&\,\begin{multlined}[t] \Vert \nabla \overline{p}_* \Vert_{L^\infty(L^4)} \Vert p_{*tt}^{(2)}\Vert_{L^2(L^4)}+\Vert  \overline{p}_* \Vert_{L^\infty(L^\infty)} \Vert \nabla p_{*tt}^{(2)}\Vert_{L^2(L^2)}\\
+ \Vert \nabla \overline{p}_{*t} \Vert_{L^2(L^4)} \Vert p_{*t}^{(1)}+p_{*t}^{(2)} \Vert_{L^\infty(L^4)}+ \Vert  \overline{p}_{*t} \Vert_{L^2(L^\infty)}\Vert \nabla p_{*t}^{(1)}+\nabla p_{*t}^{(2)} \Vert_{L^\infty(L^2)}\end{multlined} \\
\lesssim&\,\begin{multlined}[t]\Vert \Delta \overline{p}_* \Vert_{L^\infty(L^2)}\Vert \nabla p_{*tt}^{(2)}\Vert_{L^2(L^2)}+\Vert  \Delta \overline{p}_* \Vert_{L^\infty(L^2)} \Vert \nabla p_{*tt}^{(2)}\Vert_{L^2(L^2)}\\
+\Vert \Delta \overline{p}_{*t} \Vert_{L^2(L^2)}\Vert \nabla p_{*t}^{(1)}+\nabla p_{*t}^{(2)} \Vert_{L^\infty(L^2)}\\
+\Vert \Delta \overline{p}_{*t} \Vert_{L^2(L^2)}\Vert \nabla p_{*t}^{(1)}+\nabla p_{*t}^{(2)} \Vert_{L^\infty(L^2)}.\end{multlined}
\end{aligned}
\end{equation}
Hence,
\begin{equation}\label{Last_Term_nabla_f13_Main}
\begin{aligned}
&\Vert k(\Theta^{(2)}_*) \Vert_{L^\infty(L^\infty)}\Vert  \nabla (\overline{p}_*p_{*tt}^{(2)}+\overline{p}_{*t}(p_{*t}^{(1)}+p_{*t}^{(2)}))\Vert_{L^2(L^2)}\\
\lesssim&\,R_1 \Vert (p^{(1)}_* -p^{(2)}_*, \Theta^{(1)}_*-\Theta^{(2)}_*)\Vert_{X_T}. 
\end{aligned}
\end{equation}
Consequently, from the derived bounds we infer
\begin{equation}\label{nabla_f_13_Main}
\Vert\nabla \overline{f}_{13} \Vert_{L^2(L^2)}\lesssim C_T R_1  (1+R_2+R_2^{\gamma_2+2})\Vert (p^{(1)}_* -p^{(2)}_*, \Theta^{(1)}_*-\Theta^{(2)}_*)\Vert_{X_T}. 
\end{equation}
By collecting the derived estimates of separate contributions to $\overline{f}_1$, we arrive at 
\begin{equation}\label{nabla_f1_Main}
\begin{aligned}
&\Vert\nabla \overline{f}_1 \Vert_{L^2(L^2)}\\
\lesssim&\,\begin{multlined}[t] C_{T} (R_1+R_1^2)\left(1+R_2^{\gamma_1}+R_2^{\gamma_1+1}\right)
\Vert (p^{(1)}_* -p^{(2)}_*, \Theta^{(1)}_*-\Theta^{(2)}_*)\Vert_{X_T}. \end{multlined}
\end{aligned}
\end{equation}
\subsubsection*{The estimate of $\|\partial_t \overline{f}_1\|_{L^2(H^{-1})}$} Our next task is to estimate $\|\partial_t \overline{f}_1\|_{L^2(H^{-1})}$.  As above, we estimate the contributions $\|\partial_t \overline{f}_{1j}\|_{L^2(H^{-1})}$ for $j=1,2,3$ separately. We start by noting that 
\[
\begin{aligned}
\partial_t \overline{f}_{11}=&\, \begin{multlined}[t]
2\left\{k(\Theta^{(1)}_*)-k(\Theta^{(2)}_*)\right\}\left(p^{(1)}_*p_{*ttt}^{(2)}+p^{(1)}_{*t}p_{*tt}^{(2)}+2p_{*t}^{(1)}p_{*tt}^{(1)}\right)\\
+2\partial_t \left\{k(\Theta^{(1)}_*)-k(\Theta^{(2)}_*)\right\}\left(p^{(1)}_*p_{*tt}^{(2)}+(p_{*t}^{(1)})^2\right). \end{multlined}
\end{aligned}
\]
By employing the $H^{-1}$ estimate stated in \eqref{Hneg_estimate}, we then find that 
\begin{equation}\label{f_t_11}
\begin{aligned}
&\|\partial_t \overline{f}_{11}\|_{H^{-1}}\\
\lesssim&\,
\begin{multlined}[t]\left(\|k(\Theta^{(1)})_*-k(\Theta^{(2)}_*)\|_{L^\infty}+\|\nabla (k(\Theta^{(1)})_*- k(\Theta^{(2)}_*))\|_{L^3}\right)\|p^{(1)}_*p_{*ttt}^{(2)}\|_{H^{-1}}\\
+ \|k(\Theta^{(1)})_*-k(\Theta^{(2)}_*)\|_{L^\infty)}(\| p^{(1)}_{*t}p_{*tt}^{(2)} \|_{L^2}+\|p_{*t}^{(1)}p_{*tt}^{(1)}\|_{L^2})\\
+ \|\partial_t(k(\Theta^{(1)}_*)-k(\Theta^{(2)}_*))\|_{L^6}\|(p_{*t}^{(1)})^2\|_{L^3}\\
+\left\|\partial_t \left\{k(\Theta^{(1)}_*)-k(\Theta^{(2)}_*)\right\}p^{(1)}_*p_{*tt}^{(2)}\right\|_{H^{-1}}.
 \end{multlined}
\end{aligned}
\end{equation}
Hence, we obtain from above
\begin{equation}\label{f_t_11_L_2}
\begin{aligned}
&\|\partial_t \overline{f}_{11}\|_{L^2(H^{-1})}\\
\lesssim&\,\begin{multlined}[t]\left(\|k(\Theta^{(1)})_*-k(\Theta^{(2)}_*)\|_{L^\infty(L^\infty)}+\|\nabla (k(\Theta^{(1)})_*- k(\Theta^{(2)}_*))\|_{L^\infty(L^3)}\right)\|p^{(1)}_*p_{*ttt}^{(2)}\|_{L^2(H^{-1})}\\
+ \|k(\Theta^{(1)})_*-k(\Theta^{(2)}_*)\|_{L^\infty(L^\infty)}(\| p^{(1)}_{*t}p_{*tt}^{(2)} \|_{L^2(L^2)}+\|p_{*t}^{(1)}p_{*tt}^{(1)}\|_{L^2(L^2)})\\
+ \|\partial_t(k(\Theta^{(1)}_*)-k(\Theta^{(2)}_*))\|_{L^2({L^6})}\|(p_{*t}^{(1)})^2\|_{L^\infty(L^3})\\
+\left\|\partial_t \left\{k(\Theta^{(1)}_*)-k(\Theta^{(2)}_*)\right\}p^{(1)}_*p_{*tt}^{(2)}\right\|_{L^2(H^{-1})}.
 \end{multlined}
\end{aligned}
\end{equation}
 We estimate the second term by using the $H^{-1}$ inequality \eqref{Hneg_estimate} as follows:
\begin{equation}\label{H_-1_p_tt}
\begin{aligned}
\|p^{(1)}_*p_{*ttt}^{(2)}\|_{L^2(H^{-1})}\lesssim&\, \|p_{*ttt}^{(2)}\|_{L^2(H^{-1})}(\|\nabla p^{(1)}_*\|_{L^\infty(L^3)}+\|p^{(1)}_*\|_{L^\infty(L^\infty)})\\
\lesssim&\,\|p_{*ttt}^{(2)}\|_{L^2(H^{-1})}\|\Delta p^{(1)}_*\|_{L^\infty(L^2)}\\
\lesssim&\, R_1^2, 
\end{aligned}
\end{equation} 
where we have also used the embeddings $H^1(\Omega)\hookrightarrow L^3(\Omega)$, $H^2(\Omega) \hookrightarrow L^\infty(\Omega)$, and elliptic regularity. Next, as in \eqref{First_Term_f_1}, we have
\begin{equation}\label{k_1_Theta_Estimate}
\|k(\Theta^{(1)})_*-k(\Theta^{(2)}_*)\|_{L^\infty(L^\infty)}\lesssim (1+R_2^{\gamma_2+1})\Vert (\overline{p}_*, \overline{\Theta}_*)\Vert_{X_T}. 
\end{equation}
Further,
\begin{equation}\label{p_Theta_Estimate_1}
\begin{aligned}  
& \|p^{(1)}_{*t}p_{*tt}^{(2)}\|_{L^2(L^2)}+\|\partial_t(p_{*t}^{(1)})^2\|_{L^2(L^2)}\\
\lesssim&\,\|p^{(1)}_{*t}\|_{L^2(L^\infty)}\|p_{*tt}^{(2)}\|_{L^\infty(L^2)}+\|p^{(1)}_{*t}\|_{L^2(L^\infty)}\|p_{*tt}^{(1)}\|_{L^\infty(L^2)}\\
\lesssim&\,\|\Delta p^{(1)}_{*t}\|_{L^2(L^2)}\|p_{*tt}^{(2)}\|_{L^\infty(L^2)}+\|\Delta p^{(1)}_{*t}\|_{L^2(L^2)}\|p_{*tt}^{(1)}\|_{L^\infty(L^2)}
\lesssim\, R_1^2. 
\end{aligned}
\end{equation} 
Now, we can use the following re-arrangement: 
\begin{equation}
\begin{aligned}  
\partial_t(k(\Theta^{(1)}_*)-k(\Theta^{(2)}_*))=&\,k^\prime (\Theta^{(1)}_*) \Theta^{(1)}_{*t}-k^\prime (\Theta^{(2)}_*) \Theta^{(2)}_{*t}\\
=&\, k^\prime (\Theta^{(1)}_*)  (\Theta^{(1)}_{*t}-\Theta^{(2)}_{*t})+ \Theta^{(2)}_{*t}(k^\prime (\Theta^{(1)}_*)-k^\prime (\Theta^{(2)}_*)). 
\end{aligned}
\end{equation}
Hence, by the embedding $H^1(\Omega) \hookrightarrow L^{6}(\Omega)$, 
\begin{equation}\label{k_t_Theta_Estimate}
\begin{aligned}
&\|\partial_t(k(\Theta^{(1)}_*)-k(\Theta^{(2)}_*))\|_{L^2({L^6})}\\
\lesssim&\,\begin{multlined}[t] \Vert k^\prime (\Theta^{(1)}_*) \Vert_{L^\infty(L^\infty)} \Vert \Theta^{(1)}_{*t}-\Theta^{(2)}_{*t}\Vert_{L^2({L^6})}\\
+ \Vert \Theta^{(2)}_{*t} \Vert_{L^2({L^6})} \Vert k^\prime (\Theta^{(1)}_*)-k^\prime (\Theta^{(2)}_*) \Vert_{L^\infty(L^\infty)} \end{multlined} \\
\lesssim&\,\begin{multlined}[t] \Vert \Theta^{(1)}_{*t}-\Theta^{(2)}_{*t}\Vert_{L^2(H^1)}\left(1+\Vert \Theta^{(1)}_*\Vert_{L^\infty(L^\infty)}^{\gamma_2+1}\right)\\
+ \Vert \Theta^{(2)}_{*t}\Vert_{L^2(H^1)}\Vert\Theta^{(1)}_*-\Theta^{(2)}_*\Vert_{L^\infty(L^\infty)}\Big(1+\Vert \Theta^{(1)}_*\Vert_{L^\infty(L^\infty)}^{\gamma_2}\Big)\end{multlined} \\
\lesssim&\, (1+R_2+R_2^{\gamma_2+1})\Vert (\overline{p}_*, \overline{\Theta}_*)\Vert_{X_T}. 
\end{aligned}
\end{equation}
Furthermore, we have 
\begin{equation}\label{p_Theta_Estimate_2}
\begin{aligned}
\|(p_{*t}^{(1)})^2\|_{L^\infty(L^3)}
\lesssim\, \Vert p_{*t}^{(1)} \Vert_{L^\infty(L^6)}^2
\lesssim&\, \Vert \nabla p_{*t}^{(1)} \Vert_{L^\infty(L^2)}^2
\lesssim  R_1^2. 
\end{aligned}
\end{equation}
Next by using the embedding $L^{6/5}(\Omega) \hookrightarrow H^{-1}(\Omega)$ and H\"older's inequality, we infer 
\begin{equation}
\begin{aligned}
&\left\|\partial_t \left\{k(\Theta^{(1)}_*)-k(\Theta^{(2)}_*)\right\}p^{(1)}_*p_{*tt}^{(2)}\right\|_{L^2(H^{-1})}\\
\lesssim&\,\|\partial_t(k(\Theta^{(1)}_*)-k(\Theta^{(2)}_*))\|_{L^2(L^3)}\|p^{(1)}_*p_{*tt}^{(2)}\|_{L^\infty(L^2)}
\end{aligned}
\end{equation}
As in \eqref{k_t_Theta_Estimate}, using the embedding $H^1(\Omega) \hookrightarrow L^3(\Omega)$ yields 
 \begin{equation}\label{k_t_L_3}
\|\partial_t(k(\Theta^{(1)}_*)-k(\Theta^{(2)}_*))\|_{L^2(L^3)}\lesssim (1+R_2+R_2^{\gamma_2+1})\Vert (\overline{p}_*, \overline{\Theta}_*)\Vert_{X_T},
\end{equation}
whereas
\begin{equation}\label{p_Theta_Estimate_2}
\begin{aligned}
\|p^{(1)}_*p_{*tt}^{(2)}\|_{L^\infty(L^2)}
\lesssim&\, \Vert \Delta p^{(1)}_* \Vert_{L^\infty(L^2)}\Vert  p_{*tt}^{(2)} \Vert_{L^\infty(L^2)}
\lesssim  R_1^2. 
\end{aligned}
\end{equation}
Consequently, by collecting the derived estimates, we obtain from \eqref{f_t_11},
\begin{equation}\label{f_t_11_Main}
\|\partial_t \overline{f}_{11}\|_{L^2(H^{-1})}\leq C R_1^2 (1+R_2^{\gamma_2}+R_2^{\gamma_2+1}+R_2^{2\gamma_2+2})\Vert (\overline{p}_*, \overline{\Theta}_*)\Vert_{X_T}.  
\end{equation}
Next, we estimate $\overline{f}_{12}=\left\{q(\Theta^{(1)}_*)-q(\Theta^{(2)}_*)\right\}\Delta p^{(2)}_*$. We have $\|\partial_t \overline{f}_{12}\|_{L^2(H^{-1})} \lesssim \|\partial_t \overline{f}_{12}\|_{L^2(L^2)}$ and further
\begin{equation}
\begin{aligned}
\|\partial_t \overline{f}_{12}\|_{L^2(L^2)}\lesssim&\, \begin{multlined}[t] \|q(\Theta^{(1)}_*)-q(\Theta^{(2)}_*)\|_{L^\infty(L^\infty)}\| \Delta p^{(2)}_{*t}\|_{L^2(L^2)}\\
+\|\partial_t(q(\Theta^{(1)}_*)-q(\Theta^{(2)}_*))\|_{L^2(L^4)}\| \Delta p^{(2)}_*\|_{L^\infty(L^4)}. \end{multlined}
\end{aligned}
\end{equation}
Similarly to the estimate of $\|\partial_t \overline{f}_{11}\|_{L^2(L^2)}$ and by using the fact that \begin{equation}
\Vert\Delta p^{(2)}_*\Vert_{L^\infty(L^4)}\lesssim \Vert\Delta p^{(2)}_*\Vert_{L^\infty(L^2)}+\Vert\Delta\nabla  p^{(2)}_*\Vert_{L^\infty(L^2)}\lesssim \Vert p^{(2)}_*\Vert_{X_p},
\end{equation}
we obtain
\begin{equation}\label{f_t_12_Main}
\|\partial_t \overline{f}_{12}\|_{L^2(L^2)}\leq C_T R_1^2 (1+R_2+R_2^{\gamma_1+1})\Vert (\overline{p}_*, \overline{\Theta}_*)\Vert_{X_T}. 
\end{equation}
It remains to estimate $\|\partial_t \overline{f}_{13}\|_{L^2(H^{-1})}$. Indeed, recalling that
\[
\overline{f}_{13}=2k(\Theta^{(2)}_*)\left(\overline{p}_*p_{*tt}^{(2)}+\overline{p}_{*t}(p_{*t}^{(1)}+p_{*t}^{(2)})\right)
\] 
we have
 \begin{equation}\label{partial_t_f13}
 \begin{aligned}
 \|\partial_t \overline{f}_{13}\|_{L^2(H^{-1})}\lesssim&\, \|\partial_t(\overline{p}_*p_{*tt}^{(2)}+\overline{p}_{*t}(p_{*t}^{(1)}+p_{*t}^{(2)}))\|_{L^2(H^{-1})}\\
&+\|k'(\Theta^{(2)}_*) \Theta^{(2)}_{*t}\|_{L^2(L^4)}\|\overline{p}_*p_{*tt}^{(2)}+\overline{p}_{*t}(p_{*t}^{(1)}+p_{*t}^{(2)})\|_{L^\infty(L^4)}. 
\end{aligned}
\end{equation}
We estimate the first term in \eqref{partial_t_f13} as follows:
 \begin{equation}
 \begin{aligned}
&\|\partial_t(\overline{p}_*p_{*tt}^{(2)}+\overline{p}_{*t}(p_{*t}^{(1)}+p_{*t}^{(2)}))\|_{L^2(H^{-1})}\\
\lesssim &\,\begin{multlined}[t]\Vert \overline{p}_{*t} \Vert_{L^2(L^4)}\Vert p_{*tt}^{(2)}\Vert_{L^\infty(L^4)} +(\Vert \overline{p}_{*} \Vert_{L^\infty(L^\infty)}+\Vert \nabla \overline{p}_{*} \Vert_{L^\infty(L^3)})\Vert p_{*ttt}^{(2)}\Vert_{L^2(H^{-1})}\\
+ \Vert \overline{p}_{*tt} \Vert_{L^\infty(L^2)} (\Vert p_{*t}^{(1)}\Vert_{L^2(L^\infty)}+\Vert p_{*t}^{(2)}\Vert_{L^2(L^\infty)})\\
+\Vert \overline{p}_{*t} \Vert_{L^2(L^4)} (\Vert p_{*tt}^{(1)}\Vert_{L^\infty(L^4)}+\Vert p_{*tt}^{(2)}\Vert_{L^\infty (L^4)})\end{multlined} \\
\lesssim &\,\begin{multlined}[t]\Vert \nabla \overline{p}_{*t} \Vert_{L^2(L^2)}\Vert \nabla p_{*tt}^{(2)}\Vert_{L^\infty(L^2)}+\Vert \Delta \overline{p}_{*} \Vert_{L^\infty(L^2)}\Vert p_{*ttt}^{(2)}\Vert_{L^2(H^{-1})}\\  
+ \Vert \overline{p}_{*tt} \Vert_{L^\infty(L^2)} (\Vert \Delta p_{*t}^{(1)}\Vert_{L^2(L^2)}+\Vert\Delta p_{*t}^{(2)}\Vert_{L^2(L^2)})\\
+\Vert \nabla \overline{p}_{*t} \Vert_{L^2(L^2)} (\Vert\nabla  p_{*tt}^{(1)}\Vert_{L^\infty(L^2)}+\Vert \nabla p_{*tt}^{(2)}\Vert_{L^\infty (L^2)}).\end{multlined}
\end{aligned}
\end{equation}
Hence, we obtain 
\begin{equation}\label{First_Term_f13_Main}
 \begin{aligned}
\|\partial_t(\overline{p}_*p_{*tt}^{(2)}+\overline{p}_{*t}(p_{*t}^{(1)}+p_{*t}^{(2)}))\|_{L^2(H^{-1})}
\lesssim &\, R_1 \Vert (\overline{p}_*, \overline{\Theta}_*)\Vert_{X_T}. 
\end{aligned}
\end{equation}
Next, we estimate the second term on the right-hand side of \eqref{partial_t_f13} as: 
 \begin{equation}\label{Second_Term_f13_Main}
 \begin{aligned}
\|k'(\Theta^{(2)}_*) \Theta^{(2)}_{*t}\|_{L^2(L^4)}\lesssim&\, \|k'(\Theta^{(2)}_*) \|_{L^\infty(L^\infty)}\| \Theta^{(2)}_{*t}\|_{L^2(L^4)}\\
\lesssim &\, \big(1+\Vert \Theta^{(2)}_*\Vert_{L^\infty(L^\infty)}^{\gamma_2+1}\big)\| \Theta^{(2)}_{*t}\|_{L^2(H^1)}\\
\lesssim&\,R_2(1+R_2^{\gamma_2+1}). 
\end{aligned}  
\end{equation}
Finally, we estimate the last term on the right-hand side of \eqref{partial_t_f13} as 
\begin{equation}\label{Third_Term_f13}
\begin{aligned}
&\|\overline{p}_*p_{*tt}^{(2)}+\overline{p}_{*t}(p_{*t}^{(1)}+p_{*t}^{(2)})\|_{L^\infty(L^4)}\\
\lesssim &\,\begin{multlined}[t]\|\overline{p}_*\|_{L^\infty(L^\infty)}\|p_{*tt}^{(2)}\|_{L^\infty(L^4)}\\
+\|\overline{p}_{*t}\|_{L^\infty(L^4)}(\,\|p_{*t}^{(1)}\|_{L^\infty(L^\infty)}+\|p_{*t}^{(2)}\|_{L^\infty(L^\infty)})\end{multlined}\\
\lesssim &\,\begin{multlined}[t]\|\Delta \overline{p}_{*}\|_{L^\infty(L^2)}\|\nabla p_{*tt}^{(2)}\|_{L^\infty(L^2)}\\
+\|\nabla \overline{p}_{*t}\|_{L^\infty(L^2)}(\,\|\Delta p_{*t}^{(1)}\|_{L^\infty(L^2)}+\|\Delta p_{*t}^{(2)}\|_{L^\infty(L^2)}).\end{multlined}
\end{aligned}
\end{equation}
Using the embedding $H^1(0,t) \hookrightarrow  C[0,t]$, we find that 
\begin{equation}
\|\nabla \overline{p}_{*t}\|_{L^\infty(L^2)}\lesssim \|\nabla \overline{p}_{*t}\|_{L^2(L^2)}+\|\nabla \overline{p}_{*tt}\|_{L^2(L^2)}\lesssim \Vert \overline{p}_*\Vert_{X_p}. 
\end{equation}
 Consequently, we obtain from \eqref{Third_Term_f13},  
 \begin{equation}\label{Third_Term_f13_Main}
\|\overline{p}_*p_{*tt}^{(2)}+\overline{p}_{*t}(p_{*t}^{(1)}+p_{*t}^{(2)})\|_{L^\infty(L^4)}\lesssim R_1 \Vert (\overline{p}_*,\overline{\Theta}_*)\Vert_{X_T}. 
\end{equation}
Collecting \eqref{First_Term_f13_Main}, \eqref{Second_Term_f13_Main}, and \eqref{Third_Term_f13_Main} results in
\begin{equation}\label{f_t_13_Main}
  \|\partial_t \overline{f}_{13}\|_{L^2(H^{-1})}\lesssim R_1 (1+R_2+R_2^{\gamma_2+2}) \Vert (\overline{p}_*, \overline{\Theta}_*)\Vert_{X_T}. 
\end{equation}
Finally, by collecting the bounds of separate contributions, we infer that
\begin{equation}\label{f_1_t_Main_Estimate}
   \|\partial_t \overline{f}_1\|_{L^2(H^{-1})} \leq C_{T} ( R_1+R_1^2 )(1+R_2+R_2^{\gamma_2+1}+R_2^{\gamma_2+2})\Vert (\overline{p}_*, \overline{\Theta}_*)\Vert_{X_T}.  
\end{equation}
\subsubsection*{The estimate of $\|\overline{f}_2\|_{H^1(L^2)}$}
 We can bound the source term in the heat equation as follows:  
\begin{equation}\label{f_2_H_1}
\|\overline{f}_2\|_{H^1(L^2)} \lesssim \|\mathcal{Q}(p_{*t}^{(1)})-\mathcal{Q}(p_{*t}^{(2)})\|_{L^2(L^2)}+ \|\partial_t(\mathcal{Q}(p_{*t}^{(1)})-\mathcal{Q}(p_{*t}^{(2)}))\|_{L^2(L^2)}.  
\end{equation}  
Since $p_{*t}^{(j)}\in B$ for $j=1,2$, we have by the Sobolev embedding 
\begin{equation}
\Vert p_{*t}^{(j)}\Vert_{L^\infty(L^\infty)} \lesssim  \Vert \Delta p_{*t}^{(j)}\Vert_{L^\infty(L^2)}\lesssim \Vert p_{*t}^{(j)}\Vert_{L^\infty(X_p)}\lesssim R_1. 
\end{equation}
Hence, in view of the assumption \eqref{Lipschitz_Assumption}, this yields 
\begin{equation}\label{Q_Estimate}
\|\mathcal{Q}(p_{*t}^{(1)})-\mathcal{Q}(p_{*t}^{(2)})\|_{L^2(L^2)}\lesssim R_1 \Vert p_{*t}^{(1)}-p_{*t}^{(2)}\Vert_{L^2(L^2)}\lesssim R_1 \Vert p_{*t}^{(1)}-p_{*t}^{(2)}\Vert_{X_p}
\end{equation}
Similarly, using \eqref{Lipschitz_Assumption_2}, we have    
\begin{equation}\label{Q_t_Estimate}
\begin{aligned}
\|\partial_t(\mathcal{Q}(p_{*t}^{(1)})-\mathcal{Q}(p_{*t}^{(2)}))\|_{L^2(L^2)}\lesssim &\,\begin{multlined}[t]\Vert p_{*t}^{(1)}\Vert_{L^2(L^\infty)} \Vert p_{*tt}^{(1)}-p_{*tt}^{(2)}\Vert_{L^\infty(L^2)}\\
+\Vert p_{*tt}^{(2)}\Vert_{L^\infty(L^2)} \Vert p_{*t}^{(1)}-p_{*t}^{(2)}\Vert_{L^2(L^\infty)}\end{multlined}\\
\lesssim&\, \begin{multlined}[t] \Vert \Delta p_{*t}^{(1)}\Vert_{L^2(L^2)} \Vert p_{*tt}^{(1)}-p_{*tt}^{(2)}\Vert_{L^\infty(L^2)}\\
+\Vert p_{*tt}^{(2)}\Vert_{L^\infty(L^2)} \Vert \Delta (p_{*t}^{(1)}-p_{*t}^{(2)})\Vert_{L^2(L^2)}  \end{multlined}\\
\lesssim&\, R_1 \Vert p_{*t}^{(1)}-p_{*t}^{(2)}\Vert_{X_p}
\end{aligned}
\end{equation} 
Plugging \eqref{Q_Estimate} and \eqref{Q_t_Estimate} into \eqref{f_2_H_1}, we obtain 
\begin{equation}\label{partiial_t_f_2_Estimate}
\|\overline{f}_2\|_{H^1(L^2)} \lesssim  R_1\Vert (\overline{p}_*, \overline{\Theta}_*)\Vert_{X_T}. 
\end{equation}    
\subsubsection*{The energy bound for the difference equations.} Now we can apply the energy results of Proposition~\ref{Prop:LinWellp} to system \eqref{System_Contract} by setting
\begin{equation}
\alpha= 1-2k(\Theta_*^{(1)})p_*^{(1)}, \quad r=q(\Theta_*^{(1)}), \quad f_1= \overline{f}_1, \quad f_2=\overline{f}_2.
\end{equation}  
Adding the energy estimate for the pressure to the energy bound \eqref{Heat_bound} for the temperature (where now $\tilde{f}=f_2=\overline{f}_2$), we obtain
\begin{equation}\label{Contrac_1}
\begin{aligned}
&\Vert (\overline{p},\overline{\Theta})\Vert^2_{X_T}\\
=&\, \Vert \mathcal{T}(p^{(1)}_*, \Theta^{(1)}_*)-\mathcal{T}(p^{(2)}_*, \Theta^{(2)}_*)\Vert^2_{X_T}\\
\lesssim&\,\begin{multlined}[t] \int_0^t\exp{\left(\int_s^t (1+\Lambda(\sigma))\textup{d}\sigma\right)}\left(\Vert \overline{f}_1(t)\Vert_{H^1}^2+(1+\left\|\nabla \alpha(t)\right\|_{L^3}^2)\Vert \partial_t \overline{f}_1(t)\Vert_{H^{-1}}^2\right)\ds\\
+  \Vert \partial_t \overline{f}_2 \Vert_{L^2(L^2)}^2 \end{multlined}
\end{aligned}
\end{equation}
with $\Lambda=\Lambda(t)$ defined in \eqref{def_Lamba}. We have 
\begin{equation}\label{Lambda_Estimate}
\begin{aligned}
&\Vert\Lambda\Vert_{L^1(0,T)}\\
\lesssim &\,\left\Vert r_t\right\Vert_{L^1(L^2)}+\left\Vert \alpha_t\right\Vert_{L^1(L^2)}+\|\nabla r\|_{L^1(L^4)}+\Vert \alpha_t\Vert_{L^2(L^4)}^2+\Vert r_t\Vert_{L^2(L^4)}^2\\  
\lesssim&\,\begin{multlined}[t]\Vert q'(\Theta_*^{(1)})\Theta_{*t}^{(1)}\Vert_{L^1(L^2)}+\Vert k'(\Theta_*^{(1)})\Theta_{*t}^{(1)}p_*^{(1)}\Vert_{L^1(L^2)}+\Vert k(\Theta_*^{(1)})p_{*t}^{(1)}\Vert_{L^1(L^2)}\\
+\Vert q'(\Theta_*^{(1)})\nabla\Theta_{*}^{(1)}\Vert_{L^1(L^4)}
+\Vert k'(\Theta_*^{(1)})\Theta_{*t}^{(1)}p_*^{(1)}\Vert_{L^2(L^4)}^2\\
+\Vert k(\Theta_*^{(1)})p_{*t}^{(1)}\Vert_{L^2(L^4)}^2
+\Vert q'(\Theta_*^{(1)})\Theta_{*t}^{(1)}\Vert_{L^2(L^4)}^2.\end{multlined}
\end{aligned}
\end{equation}
We estimate the terms on the right-hand side of \eqref{Lambda_Estimate} as follows: using \eqref{q_prime_Assumption}, we have  
\begin{equation}\label{Estimate_Lambda_First_Term}
\begin{aligned}
\Vert q'(\Theta_*^{(1)})\Theta_{*t}^{(1)}\Vert_{L^1(L^2)}\lesssim&\, \sqrt{T}\Vert q'(\Theta_*^{(1)})\Vert_{L^\infty(L^\infty)} \Vert \Theta_{*t}^{(1)}\Vert_{L^2(L^2)}\\
\lesssim&\, \sqrt{T}(1+\Vert \Theta_*^{(1)}\Vert^{\gamma_1+1}_{L^\infty(L^\infty)}) \Vert \Theta_{*t}^{(1)}\Vert_{L^2(L^2)}\\
\lesssim&\, \sqrt{T} R_2(1+R_2^{\gamma_1+1}). 
\end{aligned}
\end{equation}
Further, by using assumption \eqref{Assumption_k_2} we have
\begin{equation}\label{Estimate_Lambda_Second_Term}
\begin{aligned}
\Vert k'(\Theta_*^{(1)})\Theta_{*t}^{(1)}p_*^{(1)}\Vert_{L^1(L^2)}\lesssim&\, \Vert k'(\Theta_*^{(1)})\Vert_{L^\infty(L^\infty)}\Vert \Theta_{*t}^{(1)}\Vert_{L^2(L^2)}\Vert p_*^{(1)}\Vert_{L^2(L^\infty)}\\
\lesssim&\,\sqrt{T}(1+\Vert \Theta_*^{(1)}\Vert^{\gamma_2+1}_{L^\infty(L^\infty)})\Vert \Theta_{*t}^{(1)}\Vert_{L^2(L^2)}\Vert \Delta p_*^{(1)}\Vert_{L^\infty(L^2)}\\
\lesssim&\,\sqrt{T}R_1(R_2+R_2^{\gamma_2+2}). 
\end{aligned}
\end{equation}
Next we find that
\begin{equation}\label{Estimate_Lambda_Third_Term}
\Vert k(\Theta_*^{(1)})p_{*t}^{(1)}\Vert_{L^1(L^2)}\lesssim T\Vert k(\Theta_*^{(1)})\Vert_{L^\infty(L^\infty)}\Vert p_{*t}^{(1)}\Vert_{L^\infty(L^2)}\lesssim T R_1,
\end{equation}
where we have used \eqref{Assumption_k_1} in the last estimate. Using the bound 
$\Vert \nabla\Theta_{*}^{(1)}\Vert_{L^4}\lesssim \|\Theta_*^{(1)}\|_{\Honetwo}$, we also have 
\begin{equation}\label{Estimate_Lambda_Fourth_Term}
\begin{aligned}
\Vert q'(\Theta_*^{(1)})\nabla\Theta_{*}^{(1)}\Vert_{L^1(L^4)}\lesssim&\, \Vert q'(\Theta_*^{(1)})\Vert_{L^\infty(L^\infty)}\Vert \nabla\Theta_{*}^{(1)}\Vert_{L^1(L^4)}\\
\lesssim&\, T(1+\Vert \Theta_*^{(1)}\Vert^{\gamma_1+1}_{L^\infty(L^\infty)})\Vert \nabla\Theta_{*}^{(1)}\Vert_{L^\infty(\Honetwo)} \\
\lesssim&\, T(R_2+R_2^{\gamma_1+2}). 
\end{aligned}
\end{equation}
Also, we have as above 
\begin{equation}\label{Estimate_Lambda_Fifth_Term}
\begin{aligned}
\Vert k'(\Theta_*^{(1)})\Theta_{*t}^{(1)}p_*^{(1)}\Vert_{L^2(L^4)}^2\lesssim&\, \Vert k'(\Theta_*^{(1)})\Vert_{L^\infty(L^\infty)}^2\Vert \Theta_{*t}^{(1)}\Vert_{L^2(L^4)}^2\Vert p_*^{(1)}\Vert_{L^\infty(L^\infty)}^2\\
\lesssim &\, (1+\Vert \Theta_*^{(1)}\Vert^{2\gamma_2+2}_{L^\infty(L^\infty)})\Vert \Theta_{*t}^{(1)}\Vert_{L^2(H^1)}^2\Vert \Delta p_*^{(1)}\Vert_{L^\infty(L^2)}^2\\  
\lesssim &\, R_1^2 R_2^2 (1+R_2^{2\gamma_2+2}). 
\end{aligned}
\end{equation} 
Further, we have the estimate 
 \begin{equation}\label{Estimate_Lambda_sixth_Term} 
 \begin{aligned}
\Vert k(\Theta_*^{(1)})p_{*t}^{(1)}\Vert_{L^2(L^4)}^2\lesssim \Vert p_{*t}^{(1)}\Vert_{L^2(L^4)}^2\lesssim &\, \Vert\nabla p_{*t}^{(1)}\Vert_{L^2 (L^2)}^2
\lesssim \, R_1^2. 
\end{aligned}
\end{equation}
Finally, we have 
\begin{equation}\label{Estimate_Lambda_Seventh_Term}
 \begin{aligned}
\Vert q'(\Theta_*^{(1)})\Theta_{*t}^{(1)}\Vert_{L^2(L^4)}^2\lesssim&\, \Vert q'(\Theta_*^{(1)})\Vert_{L^\infty(L^\infty)}^2 \Vert \Theta_{*t}^{(1)}\Vert_{L^2(L^4)}^2\\
\lesssim &\,  (1+\Vert \Theta_*^{(1)}\Vert^{2\gamma_1+2}_{L^\infty(L^\infty)})\Vert \Theta_{*t}^{(1)}\Vert_{L^2(H^1)}^2\\
\lesssim &\,R_2^2(1+R_2^{2\gamma_1+2}).
\end{aligned}
\end{equation}
Collecting the above estimates leads to
\begin{equation}\label{Lambda_Estimate_final}
\Vert\Lambda\Vert_{L^1(0,T)}\leq C(T, R_1, R_2),
\end{equation}  
where $C=C(T, R_1, R_2)$ is a positive constant that depends on $T, R_1$, and $R_2$. 

Finally, taking into account \eqref{Lambda_Estimate} and recalling \eqref{nabla_f1_Main}, \eqref{f_1_t_Main_Estimate},  \eqref{Q_Estimate}, and \eqref{partiial_t_f_2_Estimate} we obtain 
\begin{equation}
\begin{aligned}
&\Vert \mathcal{T}(p^{(1)}_*, \Theta^{(1)}_*)-\mathcal{T}(p^{(2)}_*, \Theta^{(2)}_*)\Vert_{X_T}\\
\lesssim&\, e^{C(T, R_1, R_2)} ( R_1+R_1^2 )C(T,R_2)\Vert (p^{(1)}_* -p^{(2)}_*, \Theta^{(1)}_*-\Theta^{(2)}_*)\Vert_{X_T}.
\end{aligned}  
\end{equation}
Thus, by selecting the radius $R_1>0$ sufficiently small, we can guarantee that $\mathcal{T}$ is a strict contraction in $B$.
\end{proof}
On account of Lemmas~\ref{Lemma2} and \ref{Lemma3}, an application of the contraction mapping theorem implies that there exists a unique $(p, \Theta)=\mathcal{T} (p, \Theta)$ in $B$ which solves the coupled problem. 
\subsection*{Continuous dependence on the data}
To prove continuous dependence on the data, take $(p^{(1)}, \Theta^{(1)})$ and $(p^{(2)}, \Theta^{(2)})$ to be two solutions of \eqref{coupled_problem_eq} that correspond to the initial data $(p_0^{(1)},p_1^{(1)} ,\Theta_0^{(1)})$ and $(p_0^{(2)},p_1^{(2)}, \Theta_0^{(2)})$, respectively. Similarly to the proof of contractivity, we have the following energy bound: 
\begin{equation}\label{Uniq_1}
\begin{aligned}
&\, \Vert (p^{(1)}-p^{(2)}, \Theta^{(1)}-\Theta^{(2)})\Vert_{X_T}^2\\
\lesssim&\,\begin{multlined}[t] \mathcal{E}[p^{(1)}-p^{(2)}](0)
+\mathcal{E}[\Theta^{(1)}-\Theta^{(2)}](0)\\
+ \int_0^t\exp{\left(\int_s^t (1+\Lambda(\sigma))\textup{d}\sigma\right)}\left({\Vert \overline{f}_1(t)\Vert_{H^1}^2+(1+\left\|\nabla \alpha(t)\right\|_{L^3}^2)\Vert \partial_t \overline{f}_1(t)\Vert_{H^{-1}}^2}\right)\ds\\
+  C_T \Vert \partial_t \overline{f}_{2} \Vert_{L^2(L^2)}^2. \end{multlined}
\end{aligned}
\end{equation}
Here $\overline{f}_1$ and $\overline{f}_2$ are functions of $p^{(1)}=p^{(1)}_*$ and $p^{(2)}=p^{(2)}_*$; see \eqref{overline_f1} and \eqref{overline_f2} for their definitions. Following the same steps as in the proof of contractivity, we can deduce that there exists a function $\Psi$ that depends on $\Vert p^{(j)}\Vert_{X_p}$ and $\Vert \Theta^{(j)}\Vert_{X_p}$ with $j=1,2$, such that 
\begin{equation}\label{Uniq_1}
\begin{aligned}
&\, \Vert (p^{(1)}-p^{(2)}, \Theta^{(1)}-\Theta^{(2)})\Vert_{X_T}^2\\
\lesssim&\,\begin{multlined}[t]
\Vert (p^{(1)}_0-p^{(2)}_0, \Theta^{(1)}_0-\Theta^{(2)}_0)\Vert_{X_T}^2\\
+\int_0^t \Psi (\Vert p^{(1)}\Vert_{X_p}, \Vert p^{(2)}\Vert_{X_p}, \Vert \Theta^{(1)}\Vert_{X_\Theta},\Vert \Theta^{(2)}\Vert_{X_\Theta})\Vert (p^{(1)}-p^{(2)}, \Theta^{(1)}-\Theta^{(2)})\Vert_{X_T}^2\ds.
\end{multlined}
\end{aligned}
\end{equation}
An  application of Gronwall's inequality leads to 
\begin{equation}
\begin{aligned}
&\Vert (p^{(1)}-p^{(2)}, \Theta^{(1)}-\Theta^{(2)})\Vert_{X_T}^2\\
\lesssim&\, \Vert (p^{(1)}_0-p^{(2)}_0, \Theta^{(1)}_0-\Theta^{(2)}_0)\Vert_{X_T}^2\exp\Big\{\int_0^T\Psi (t)\dt\Big\}.
\end{aligned}
\end{equation}
This last inequality yields the desired result, which also implies uniqueness in $X_T$ by taking the data to be the same.
\end{proof} 
\section*{Conclusion and Outlook}
In this work, we have analyzed the coupled Westervelt--Pennes model of HIFU-induced heating. By relying on the energy analysis of a linearized problem and a subsequent fixed-point argument, we proved the local-in-time well-posedness of this model under the assumption of smooth and (with respect to pressure) small data. \\
\indent Although our well-posedness result does not require any smallness assumption on $T$, it is not a global existence result due to the possible dependence of data on the final time. A global existence result would require to show that there is a universal neighborhood of the origin in the topology induced by the norm of the initial conditions for which the solution exists and is bounded uniformly in time. To achieve this, we would need to refine our energy estimates by constructing compensating functionals to capture the dissipation of appropriate components of the norm of the solution. This analysis will be the subject of future research. \\
\indent We note further that in the energy estimates in Section~\ref{Sec_Energy_Analysis}, $b$ must be a positive constant, independent of $\Theta$. To permit more realistic modeling scenarios, future analysis will involve studying the case $b=b(\Theta)$ and allowing for (time- or space-) fractional damping in the model.
\bibliography{references}{}
\bibliographystyle{siam} 
\end{document}